\documentclass[a4paper,10pt]{article}
\usepackage[utf8]{inputenc}
 \usepackage[T1]{fontenc}
 \usepackage[normalem]{ulem}
 \usepackage[english]{babel}
\usepackage{graphicx}
\usepackage{amsmath}
\usepackage{amssymb}
\usepackage{amsthm}
\usepackage{graphicx}
\usepackage{amsfonts}
\usepackage{xr}
\usepackage{color}
\usepackage{graphicx,epsfig}
\usepackage{pdfsync,subfigure,enumerate,footnote}

\numberwithin{equation}{section}

\newcommand{\supp}{\mathrm{supp}\,}

\newcommand{\sign}{\mathrm{sign}\,}

\newcommand{\R}{\mathbb{R}}
\newcommand{\C}{\mathbb{C}}
\newcommand{\N}{\mathbb{N}}

\newcommand{\Y}{\mathcal{Y}}

\newcommand{\e}{\varepsilon}
\renewcommand{\epsilon}{\varepsilon}
\newcommand{\eps}{\varepsilon}
\newcommand{\sub}{h}
\newcommand{\bT}{\mathbf{T}}
\newcommand{\ug}{\underline g}
\newcommand{\ds}{\displaystyle}

\newtheorem{thm}{\bf Theorem}[section]
\newtheorem{prop}[thm]{\bf Proposition}
\newtheorem{rmq}[thm]{\bf Remark}
\newtheorem{lem}[thm]{\bf Lemma}

\newtheorem{defi}[thm]{Definition}

\newtheorem{cor}[thm]{Corollary}

\newcommand{\eb}[1]{{\color{blue}#1}}

\date\today
\author{Emeric Bouin\footnote{Corresponding author}\,
\footnote{Ecole Normale Sup\'erieure de Lyon, UMR CNRS 5669 'UMPA',  and INRIA Alpes, project-team NUMED,  46 all\'ee d'Italie, F-69364~Lyon~cedex~07, France. E-mail: \texttt{emeric.bouin@ens-lyon.fr}}\;, 
Vincent Calvez\footnote{Ecole Normale Sup\'erieure de Lyon, UMR CNRS 5669 'UMPA', and INRIA Alpes, project-team NUMED, 46 all\'ee d'Italie, F-69364~Lyon~cedex~07, France. E-mail: \texttt{vincent.calvez@ens-lyon.fr}}\;, 
Gr\'egoire Nadin\footnote{Universit\'e Pierre et Marie Curie-Paris 6, UMR CNRS 7598 'LJLL', BC187, 4 place de Jussieu, F-75252~Paris~cedex~05, France. E-mail: \texttt{nadin@ann.jussieu.fr}}}

\title{Propagation in a kinetic reaction-transport equation:\\ \eb{travelling waves and accelerating fronts}}
\begin{document}
\maketitle
\begin{abstract}
In this paper, we study the existence and stability of travelling wave solutions of a kinetic reaction-transport equation. 
The model describes particles moving according to a velocity-jump process, and proliferating thanks to a  
reaction term of monostable type. The boundedness of the velocity set appears to be a necessary and sufficient condition 
for the existence of positive travelling waves. The minimal speed of propagation of waves is obtained from an explicit dispersion relation. 
We construct the waves  using a technique of   sub- and supersolutions    and prove their \eb{weak} stability in a weighted $L^2$ space. In case of an unbounded velocity set, we prove a superlinear spreading. It appears that the rate of spreading depends  on the decay at infinity of the \eb{velocity distribution}. 
\eb{In the case of a Gaussian distribution, we prove that the front spreads as $t^{3/2}$.} 
\end{abstract}


\section{Introduction}


We address the issue of front propagation in a reaction-transport equation of kinetic type, 
\begin{equation} \label{eq:kinKPP}
\left\{\begin{array}{ll}
\partial_t g + v\partial_x g = \eb{ \left(M(v) \rho_g - g\right) + r\rho_g \left( M(v) - g\right)}  \,,\quad & (t,x,v) \in \R_+ \times \R \times V\, ,\smallskip\\
g(0,x,v) = g^0(x,v)\,, \quad & (x,v) \in   \R \times V\, .
\end{array}\right.
\end{equation}
Here, the density $g(t,x,v)$ describes a population of individuals in a continuum setting, and $\rho_g(t,x) = \int_V g(t,x,v)\, dv$ is the macroscopic density. The subset $V\subset \R$ is the set of all possible velocities. Individuals move following a velocity-jump process: they run with speed $v\in V$, and change velocity at rate 1. They instantaneously choose a new velocity with the probability distribution $M$.  Unless otherwise stated, we assume in this paper that  $V$ is symmetric and  $M$ satisfies the following properties: $M\in L^1(V) \cap\mathcal{C}^0(V) $, and
\begin{equation}\label{eq:hypM}
\int_{V} M(v)dv = 1\,, \quad \int_{V} v M(v) dv = 0\,, \quad \int_{V}v^2 M(v) dv = D<+\infty\,.
\end{equation}
In addition, individuals are able to reproduce, with rate $r>0$. New individuals start with a velocity chosen at random with the same probability distribution $M$. We could have chosen a different distribution without changing the main results, but we keep the same for the sake of clarity. Finally, we include a quadratic saturation term,  which accounts for local competition between individuals, regardless of their speed.

The main motivation for this work comes from the study of pulse waves in bacterial colonies of 
{\em Escherichia coli} \cite{Adler66,Keller-Segel71,Saragosti1,Saragosti2}. Kinetic models have been proposed to describe the run-and-tumble motion 
of individual bacteria at the mesoscopic scale \cite{Alt80,ODA88}. Several works have been dedicated to derive macroscopic equations from those kinetic 
models in the diffusion limit \cite{Hillen-Othmer00,Erban-Othmer04,Chalub04,Saragosti1}. Recently it has been shown that for some set of experiments, 
the diffusion approximation is not valid, so one has to stick to the kinetic description at the mesoscopic scale to closely compare with data \cite{Saragosti2}. 

There is one major difference between this motivation and model \eqref{eq:kinKPP}. 
Pulse waves  in bacterial colonies of {\em E. coli} are mainly driven by chemotaxis which generates macroscopic fluxes. 
Growth of the population can be merely ignored in such models. In model  \eqref{eq:kinKPP} however, growth and dispersion are the main reasons for front propagation, 
and there is no macroscopic flux due to the velocity-jump process since the distribution $M$ satisfies $\int_{V} v M(v) dv = 0$. For the sake of applications, 
we also refer to the growth and branching of the plant pathogen {\em Phytophthora} by mean of a reaction-transport equation similar to \eqref{eq:kinKPP} \cite{Phytophthora}.

There is a strong link between \eqref{eq:kinKPP} and the classical Fisher-KPP equation \cite{Fisher37,KPP}. In case of a suitable balance between scattering and growth (more scattering than growth), we can perform the parabolic rescaling $(r,t,x) \mapsto \left(\eps^2 r, \frac{t}{\eps^2} , \frac{x}{\eps} \right)$  in \eqref{eq:kinKPP},
\begin{equation} \label{eq:kinKPP2}
\epsilon^2 \partial_t g_\eps +  \epsilon  v \partial_x g_\eps  = \left(M(v) \rho_{g_\eps} - g_\eps\right) + \epsilon^2  r\rho_{g_\eps} \left( M(v) - g_\eps \right)\, .
\end{equation}
The diffusion limit yields $g_\eps\to M(v) \rho_0$, where $\rho_0$ is solution to the Fisher-KPP equation (see \cite{Cuesta12} for example), 
\begin{equation}\label{KPP}
\partial_t \rho_0- D \partial_{xx} \rho_0 = r \rho_0 \left( 1 - \rho_0 \right)\, . 
\end{equation}
We recall that for nonincreasing initial data decaying sufficiently fast at $x = +\infty$, the solution of (\ref{KPP}) behaves asymptotically as a travelling front moving at the minimal speed $c^* = 2 \sqrt{rD}$  \cite{KPP,Aronson-Weinberger75}. In addition, this front is stable in some weighted $L^2$ space \cite{Kirchgassner92,Gallay94}. Therefore it is natural to address the same questions for \eqref{eq:kinKPP}. We give below the definition of a travelling wave for equation \eqref{eq:kinKPP}. 

{\begin{defi}\label{def:deftw}
A function $g(t,x,v)$ is a smooth travelling wave solution of speed $c \in \R_+$ of equation \eqref{eq:kinKPP} if it can be written $g(t,x,v) = f \left( x - ct , v \right)$, where  the profile  $f \in \mathcal{C}^2 \left( \R \times V \right)$ satisfies \begin{equation} \label{eq:deftw}\forall (z,v) \in\R\times V\,, \quad  0\leq f(z,v)\leq M(v)\,, \quad \lim_{z\to -\infty} f(z,v)=M(v)\,, \quad \lim_{z\to +\infty} f(z,v)=0\;.
\end{equation}
\end{defi}}

In fact, $f$ is a solution of the stationary equation in the moving frame $z = x - ct$, for some $c\geq0$,
\begin{equation} \label{eqkinwave}
 ( v - c  )\partial_{z} f = \left(M(v) \rho_f - f\right) + r \rho_f \left( M(v) - f\right)\,, \quad (z,v) \in\R\times V\,.
\end{equation} 



The existence of travelling waves in reaction-transport equations has been adressed by Schwetlick \cite{Schwetlick00,Schwetlick02} 
for a similar class of equations. First, the set $V$ is bounded and  $M$ is the uniform distribution over $V$. Second, 
the nonlinearity can be chosen more generally (either monostable as here, or bistable), 
but it depends only on the macroscopic density $\rho_g$ \cite[Eq. (4)]{Schwetlick00}.  
For the monostable case, using a quite general method he has proved the  existence of travelling waves of speed {$c$} for any {$c \in [c^*,\sup V)$}, 
a result very similar to the Fisher-KPP equation.  
We emphasize that, although the equations differ  between his work and ours, 
they coincide in the linearized regime of low density $g\ll 1$. On the contrary to Schwetlick, 
we do not consider a general nonlinearity and we restrict to the logistic case, but we consider general velocity kernels $M(v)$.


More recently, the rescaled equation (\ref{eq:kinKPP2}) has been  investigated by Cuesta, Hittmeir and Schmeiser  \cite{Cuesta12}  in the parabolic
regime $\e \ll 1$. Using a micro-macro decomposition, they construct possibly  oscillatory    travelling waves of speed $c\geq 2\sqrt{rD}$ for $\e$ small enough (depending on $c$). In addition, when the set of admissible speeds $V$ is bounded, $c> 2\sqrt{rD}$, and $\e$ is small enough, they prove that the travelling wave constructed in this way is indeed nonnegative.

Lastly, when $M$ is the measure 
$M = \frac{1}{2} (\delta_{ -\nu  }+\delta_{  \nu  })$   for some $\nu>0$, equation (\ref{eq:kinKPP}) is analogous to the reaction-telegraph equation for the macroscopic density $\rho_g$ (up to a slight change in the nonlinearity however). This equation has been the subject  of  a large number of studies \cite{DO,Hadeler,Holmes93,GallayRaugel,Mendez96,Fedotov98,Fedotov99,Fort.Mendez99,Mendez04,Mendez-book}. Recently, the authors proved the existence of a minimal speed $c^*$ such that travelling waves exist for all speed $c\geq c^*$ \cite{Bouin-Calvez-Nadin}. Moreover these waves are stable in some $L^2$ weighted space, with a weight which differs from the classical exponential weight arising in the stability theory of the Fisher-KPP equation, see {\em e.g.} \cite{Kirchgassner92}.
As the reaction-telegraph equation involves both parabolic and hyperbolic contributions, the smoothness of the wave depends on the balance between these contributions. In fact there is a transition between a parabolic (smooth waves) and a hyperbolic regime (discontinuous waves), see Remark \ref{rem:two velocities} below. The authors also prove the existence of supersonic waves, having speed $c> \nu$ (see Remark \ref{rem:supersonic}).

\medskip

The aim of the present paper is to investigate the existence and stability of travelling waves for equation (\ref{eq:kinKPP}) 
for arbitrary kernels $M$ satisfying (\ref{eq:hypM}). For the existence part, we shall use the method of sub- and supersolutions, which do not rely 
on a perturbation argument. The stability part relies on the derivation of a suitable weight from which we can build a Lyapunov functional for the linearized version of \eqref{eq:kinKPP}.
The crucial assumption for the existence of travelling waves is the boundedness of $V$. We prove in fact that under the condition $(\forall v\in \R)\; M(v)>0$, there cannot exist a positive travelling wave. We finally investigate the spreading rate \eb{when $M$ is a Gaussian distribution}. 

In the last stage of writting of this paper, we realized that similar issues were formally addressed by M\'endez et al. for a slightly different equation admitting the same linearization near the front edge   \cite{Mendez10}. Our results are in agreement with their predictions.



\subsection*{Existence of travelling waves when the velocity set is bounded.}

\begin{thm} \label{existence-tw} 
Assume that the set $V$ is compact, and that $M\in \mathcal{C}^0(V)$ satisfies \eqref{eq:hypM}. Let $v_{\rm max} = \sup V$.
There exists a speed $c^* \in (0,{ v_{\rm max} })$ such that  for all $c \in [ c^* , {v_{\rm max} } )$,  there exists a travelling wave $f(x-ct,v)$ solution of \eqref{eqkinwave} with speed $c$. 
The travelling wave is nonincreasing with respect to the space variable: $\partial_z f \leq 0$. Moreover, if $\inf_V M >0$ then there exists no positive travelling wave of speed $c\in [0,c^*)$.
\end{thm}

The minimal speed $c^*$ is given through the following implicit dispersion relation. First, we observe that, for each $\lambda>0$, there is a unique $c(\lambda)\in (v_{\max} - \lambda^{-1},v_{\max})$ such that 
\begin{equation}\label{eq:dispersion-intro}
(1 + r) \int_V \frac{M(v)}{1 +\lambda(c(\lambda) - v)}\, dv = 1\, .
\end{equation}
Then we have the formula
\[ c^*= \inf_{\lambda>0} c(\lambda)\, .  \]

\begin{rmq}\label{rem:two velocities}
In the special case of two possible velocities only  \cite{Bouin-Calvez-Nadin}, corresponding to $M(v) = \frac12 \left(\delta_{-v_{\rm max}}+ \delta_{ v_{\rm max} }\right) $, two regimes have to be distinguished, namely $r<1$ and $r\geq1$. In the case $r\geq 1$ the travelling wave with minimal speed vanishes on a half-line. There, the speed of the wave is not characterized by the linearized problem for $f\ll 1$. Note that this case  is not contained in the statement of Theorem \ref{existence-tw} since it is assumed that $M\in \mathcal{C}^0(V)$. This makes a clear difference between the case of a measure $M$ which is absolutely continuous with respect to the Lebesgue measure, and the case of a  measure with atoms. 
\end{rmq}

\begin{rmq}\label{rem:supersonic}
We expect that travelling waves exist for any $c\geq c^*$, although this seems to contradict the finite speed of propagation when $c>v_{\rm max}$. In fact    supersonic  waves corresponding to $c>v_{\rm max}$ should be driven by growth mainly, as it is the case in a simplified model with only two speeds \cite{Bouin-Calvez-Nadin}. A simple argument to support the existence of such waves consists in eliminating the transport part, and seeking waves driven by growth only, 
$- c\partial_z f = M(v)\rho_f - f + r\rho_f(M -f)$. Integrating with respect to $v$ yields a logistic equation for $\rho_f$, $- c\partial_z \rho_f = r\rho_f(1 - \rho_f)$, which as a solution connecting $1$ and $0$ for any positive $c$. However these waves are quite artificial and we do not address this issue  further. 
\end{rmq}

We now define $c^*=c^*(M)$ and investigate the dependence of the minimal speed upon $M$.
In the following Proposition, we  give some general bounds on the minimal speed. 
\begin{prop} \label{prop:bound on c*} Under the same conditions as Theorem \ref{existence-tw}, assume in addition that $M$ is symmetric. Then, the minimal speed satisfies the following properties,
\begin{enumerate}[a-]
\item {\em [Scaling]} For $\sigma>0$, define $M_\sigma (v)= \sigma^{-1} M\left(\sigma^{-1}v \right)$, \eb{and rescale the velocity set accordingly $(V_\sigma := \sigma V)$,} then 
\begin{equation*}
c^*(M_\sigma)= \sigma c^*(M).
\end{equation*}
\item {\em [Rearrangement]} Denote by $M^\star$ the  Schwarz decreasing rearrangement of the function $M$   \cite{Lieb-Loss} 
and $M_\star  = -\left( -M \right)^\star$ the  Schwarz increasing rearrangement of the density distribution $M$, then
\begin{equation*}
c^*(M^\star)\leq c^* (M) \leq c^*(M_\star).
\end{equation*}
\item {\em [Comparison]}  If $r<1$ then 
 \begin{equation*}
\frac{2\sqrt{rD}}{1 + r} \leq c^* (M) \leq \frac{2\sqrt{r}}{1 + r}v_{\rm max} \, ,
\end{equation*}
whereas,  if $r\geq1$ then 
\begin{equation*}
 \sqrt{D} \leq c^* (M) \leq  v_{\rm max}  \, ,
\end{equation*}
\item {\em [Diffusion limit]} In the diffusion limit $(r,t,x) \mapsto \left(\eps^2 r, \frac{t}{\eps^2} , \frac{x}{\eps} \right)$, \eb{the dispersion relation for the rescaled equation \eqref{eq:kinKPP2} reads
\begin{equation}
(1 + \epsilon^2 r)  \int_V \dfrac{ M(v)}{1 + \epsilon^2\lambda(c - v/\eps)} \, dv\, = 1.
\end{equation}
\eb{We recover the KPP speed of the wave in the diffusive limit},
\[\lim_{\e\to 0} c^*_\e = 2\sqrt{rD} \]}
\end{enumerate}
\end{prop}

\subsection*{Spreading of the front.}

In the case where $V$ is compact, we prove that for suitable initial data $g(0,x,v)$, the front spreads asymptotically with speed $c^*$, in a weak sense.

\begin{prop} \label{prop:spreadingbounded}
Under the same conditions as Theorem \ref{existence-tw}, assume in addition that $\inf_V M>0$. Let $g^0 \in L^\infty (\R\times V)$ such that $0\leq g^0 (x,v)\leq M(v)$ for all $(x,v) \in \R\times V$. 
Let $g$ be the solution of the Cauchy 
problem \eqref{eq:kinKPP}. Then 
\begin{enumerate}
 \item if there exists $x_R$ such that
$g^0 (x,v) = 0$ for all $x\geq x_R$ and $v\in V$, then for all $c>c^*$,  
 $$(\forall v\in V)\quad \lim_{t\to +\infty} \left( \sup_{x\geq ct}  g(t,x,v)\right) =0\,,$$
\item if there exists $x_L$ and $\gamma \in (0,1)$ such that $g^0 (x,v) \geq \gamma M(v)$ for all $x\leq x_L$ and $v\in V$, then 
for all $c<c^*$,
$$(\forall v\in V)\quad \lim_{t\to +\infty} \left( \sup_{x\leq ct}  |M(v) - g(t,x,v)|\right) =0\,,$$
\end{enumerate}
where $c^*$ is the minimal speed of existence of travelling waves given by Theorem \ref{existence-tw}. 
\end{prop}

\subsection*{Stability of the travelling waves.}

We also establish linear and nonlinear stability \eb{of the travelling wave of speed $c\in [c^*, v_{\rm max})$} in some weighted $L^2$ space. The key point is to derive a suitable weight which enables to build a Lyapunov functional for the linear problem. \eb{The weight $\phi(z,v)$ is constructed in a systematic way, following \cite{Bouin-Calvez-Nadin}. However, we believe it is not optimal, as opposed to \cite{Bouin-Calvez-Nadin}, for some technical reason (see Remark \ref{rk:phi not optimal}).}

Let $f$ be a travelling wave \eqref{eqkinwave} of speed \eb{$c \in [c^*, v_{\rm max})$}, and let $u=g-f$ (resp. $u^0=g^0-f$) be  the perturbation of $f$ in the moving frame. 
Neglecting the nonlinear contributions, we are led to investigate the linear equation
\begin{equation}\label{kinmf-intro}
\partial_t u + (v - c) \partial_{ z } u + \left( 1 + r \rho_f \right) u = \left( (1+r) M  - r f   \right)  \rho_u\, .
\end{equation}
\eb{
\begin{thm}[Linear stability]\label{LinStability}

There exists a weight $\phi(z,v)$ such that the travelling front of speed $c\in [c^*, v_{\rm max})$ is linearly stable in the weighted space 
$L^2 \left( e^{- 2 \phi( z ,v)} d z  dv \right)$ in the following sense: if $u^0 \in L^2 \left( e^{- 2 \phi( z ,v)} d z  dv \right)$, then 
\begin{equation*}
(\forall t \in \R^+) \quad \Vert u(t) \Vert_{L^2 \left( e^{- 2 \phi } \right)} \leq \Vert u^0 \Vert_{L^2 \left( e^{- 2 \phi }  \right)}. 
\end{equation*}
Moreover, the perturbation is globally decaying as the dissipation is integrable in time:
\begin{equation*}
(\forall z_0 \in \R)\quad  \int_{\{ z < z_0 \} \times V} \left\vert u(t,z,v) \right\vert^2 e^{- 2 \phi( z ,v)} dzdv + \int_{\{ z > z_0 \} \times V} \rho_f(z) \left\vert u(t,z,v) \right\vert^2 e^{- 2 \phi( z ,v)} dzdv   \in L^1\left( \R^+ \right).
\end{equation*}
\end{thm}
The proposition will appear as a corollary of the following Lyapunov identity, which holds true for any solution $u$ of the linear equation \eqref{kinmf-intro},
\begin{equation}\label{eq:dissipation linear}
\frac{d}{dt} \left( \frac12 \int_{\R \times V} \left\vert u \right\vert^2 e^{- 2 \phi( z ,v)} d  z  dv \right) + \int_{\R \times V} \frac r2 \left(  \rho_f + \frac{ f}{M(v) + r \left( M(v) -  f \right)} \right) \left\vert u \right\vert^2 e^{- 2 \phi( z ,v)} d  z  dv \leq 0\,.
\end{equation}
The weight $\phi$ is given in Definition \ref{eq:def weight phi}. It is  equivalent to $- z$ as $z\to +\infty$, uniformly with respect to $v$.   

The weighted energy estimate \ref{eq:dissipation linear} does not provide any exponential decay, because of the presence of $\rho_f(z)$ in the dissipation. This is a general concern for reaction-diffusion equations, see \cite{Kirchgassner92} and references therein. However, in \cite{Cuesta12} the authors prove such an exponential decay in the case of supercritical speeds $c>2\sqrt{rD}$, and $\e$ small enough (diffusive regime). We do not follow this argument further in this work.}

Then we adapt the method of \cite{Cuesta12}, using a comparison argument together with the explicit formula of the dissipation \eqref{eq:dissipation linear}, in order to prove a nonlinear stability result.

\begin{cor}[Nonlinear stability]\label{NonLinStability}

Under the same conditions as Theorem  \ref{LinStability}, assume in addition that there exists $\gamma \in \left( \frac12 , 1 \right] $ such that
\begin{equation}\label{initialcond}
\left(\forall \left( x , v \right) \in \R \times V\right) \quad g^0(x , v) \geq \gamma  f(x, v)\,.
\end{equation}
Then the same conclusion as in Theorem \ref{LinStability} holds true.
\end{cor}

We expect that nonlinear stability holds true for any $\gamma\in (0,1]$. However this would require to redefine the weight $\phi$, since we believe it is not the optimal one,  see Remark \ref{rk:phi not optimal} below.


\subsection*{Superlinear propagation when velocity is unbounded.}

Boundedness of $V$ is a crucial hypothesis in order to build the travelling waves. 
We believe that it is a necessary and sufficient condition. We make a first step to support this conjecture by investigating the case $V = \R$.
 We first prove infinite speed of spreading of the front under the natural assumption $(\forall v\in \R)\; M(v)>0$. As a corollary there cannot exist travelling wave
in the sense of Definition \ref{def:deftw}. Note that there exist travelling waves with less restrictive conditions than Definition \ref{def:deftw}, 
at least in the diffusive regime \cite{Cuesta12}. These fronts are expected not to verify the nonnegativity condition, as $x\to +\infty$. We believe that such oscillating fronts do exist far from the diffusive regime. 
In the case where $V = \R$ and  $M$ is a Gaussian distribution, we have plotted the  dispersion relation \eqref{eq:dispersion-intro} in the complex plane $\lambda \in \C$, for an arbitrary given $c>0$. We have observed that it selects two complex conjugate roots,  supporting the fact that damped oscillating fronts should exist (results not shown).

\begin{prop} \label{prop:spreadingunbounded}
Assume that $M(v) > 0$ for all $v\in \R$. Let $g^0 \in L^\infty (\R\times V)$ such that $0\leq g^0 (x,v)\leq M(v)$ for all $(x,v) \in \R\times V$
and there exists $x_L$ and $\gamma\in (0,1)$ such that $g^0 (x,v) \geq \gamma M(v)$ for all $x\leq x_L$ and $v\in V$. 
Let $g$ be the solution of the Cauchy 
problem \eqref{eq:kinKPP}. Then 
for all $c>0$,
\begin{equation*}
(\forall v\in V)\quad \lim_{t\to +\infty}\left( \sup_{x\leq ct}  \left\vert M(v) - g(t,x,v) \right\vert\right) =0\, .
\end{equation*}
\end{prop}

We can immediately deduce from this result the non-existence of travelling waves when $V=\R$, by taking such a travelling wave as an initial datum $g^0$ 
in order to reach a contradiction.

\begin{cor}
 Assume that $M(v)>0$ for all $v\in \R$. Then equation \eqref{eq:kinKPP} does not admit any travelling wave solution. 
\end{cor}

\eb{\subsection*{Accelerating fronts for a Gaussian distribution.}

Accelerating fronts in reaction-diffusion equations have raised a lot of interest in the recent years. They occur for the Fisher-KPP equation \eqref{KPP} when the initial datum decays more slowly than any exponential \cite{Hamel-Roques}. They also appear when the diffusion operator is replaced by a nonlocal dispersal operator with fat tails \cite{Kot-Lewis,Medlock-Kot,Garnier}, or by a nonlocal fractional diffusion operator \cite{Cabre-Roquejoffre, Cabre-Roquejoffre2, Coulon-Roquejoffre}. Recently, accelerating fronts have been conjectured to occur in a reaction-diffusion-mutation model which generalizes the Fisher-KPP equation to a population structured with respect to the diffusion coefficient \cite{toads}.

Here, we investigate the case of a Gaussian distribution $M$. The spreading rate  $\langle x\rangle = \mathcal O (t^{3/2} )$ is expected in this case (heuristics, and see \cite{Mendez10}). 
We prove that spreading occurs
with this rate. For this purpose, we build suitable sub-  and 
supersolution which spread with this rate.


We split our results into two parts, respectively the upper bound and the lower bound of the spreading rate. The reason is that the constructions are quite different. The construction of the supersolution relies on a first guess inspired from \cite{Garnier}, plus convolution tricks which are made easier in the gaussian case. On the other hand, the construction of the subsolution is based on a better comprehension of the growth-dispersion process. Again, some technical estimates are facilitated in the gaussian case. We believe that these results can be generalized to a large class of distributions $M$, at the expense of clarity.

\begin{thm}\label{thm:unbounded}
Let $M(v) = \frac1{\sigma \sqrt{2 \pi}} \exp \left( - \frac{v^2}{2 \sigma^2} \right) $, defined for $v\in \R$. Let $g^0 \in L^\infty (\R\times V)$ such that $0\leq g^0 (x,v)\leq M(v)$ for all $(x,v) \in \R\times V$. We have the two following, independent, items,
\begin{enumerate}
\item Assume that there exist $1\leq b \leq a$ such that 
\begin{equation*}
(\forall (x,v)\in \R\times V)\quad g^0 (x,v) \leq \frac{1}{b} M\left( \frac{x}{b} \right) M(v) e^{ra}\,.
\end{equation*}
Let $g$ be the solution of the Cauchy 
problem \eqref{eq:kinKPP}. Then 
for all $\eps > 0$, one has 
\begin{equation*}
\lim_{t \to + \infty}\left(\sup_{ |x|\geq \left( 1 + \eps \right)\sigma \sqrt{2r} t^{3/2}}\rho_g(t, x ) \right) =  0 \,.
\end{equation*} 
\item Assume that there exists $\gamma \in (0,1)$, and $x_L\in \R$ such that
\begin{equation*}
(\forall (x,v)\in \R\times V)\quad g^0 (x,v)  \geq \gamma M(v) {\bf 1}_{x < x_L},
\end{equation*}
Let $g$ be the solution of the Cauchy problem \eqref{eq:kinKPP}. Then for all $\eps > 0$, one has 
\begin{equation*}
\lim_{t \to + \infty}\left(\sup_{ x \leq \left( 1 - \eps \right)\sigma ( \frac{r}{r + 2} t )^{3/2}}\rho_g(t, x ) \right) \geq  1 - \gamma \,.
\end{equation*} 
\end{enumerate}
\end{thm}
}

\begin{rmq}[Front propagation and diffusive limit]
There is some subtlety hidden behind this phenomenon of infinite speed of spreading. In fact the diffusion limit of the scattering equation (namely $r = 0$) towards the heat equation makes no difference between bounded or unbounded velocity sets, as soon as the variance $D$ is finite  (see \cite{Degond-Goudon-Poupaud} and the references therein). However very low densities behave quite differently, which can be measured in the setting of large deviations or WKB limit. This can be observed even in the case of a bounded velocity set. In  \cite{Bouin-Calvez} the large deviation limit of the scattering equation is performed. It differs from the classical eikonal equation obtained from the heat equation. The case of unbounded velocities is even more complicated \cite{Bouin-Calvez-Grenier-Nadin}. To conclude, let us emphasize that low densities are the one that drive the front here (pulled front). So the diffusion limit is irrelevant in the case of unbounded velocities, since very low density of particles having very large speed makes a big contribution. 
\end{rmq}

%
%
%
%
%
%
%
%

\section{Preliminary results}

We first recall some useful results concerning the Cauchy problem associated with \eqref{eq:kinKPP}: well-posedness and a strong maximum principle.
 These statements extend some results given in \cite{Cuesta12}. 
They do not rely on the boundedness of $V$. 


\begin{prop} [Global existence: Theorem 4 in \cite{Cuesta12}] \label{Cauchypbm} 
Let $g^0$ a measurable function such that $0\leq g^0(x,v)\leq M(v)$ for all $(x,v) \in \R \times V$. Then the Cauchy problem  \eqref{eq:kinKPP} has a unique solution $g\in\mathcal{C}^ 0_b( \R_+\times\R\times V)$ in the sense of distributions, satisfying
\begin{equation*}
(\forall  (t,x,v)\in \R_+\times\R\times V)\quad 0\leq g(t,x,v)\leq M(v)  \, .
\end{equation*}
\end{prop}

The next result refines the comparison principle of \cite{Cuesta12} in order to extend it to sub and supersolutions 
in the sense of distributions and to state a strong maximum principle. Its proof is given in Appendix. 

\begin{prop}[Comparison principle]\label{prop-mp}
Assume that $u_1,u_2\in\mathcal{C}(\R_+,L^\infty(\R\times V))$ are respectively a super- and a subsolution of \eqref{eq:kinKPP}, {\em i.e.} 
\begin{equation*}\begin{array}{l}
\partial_t g_1+  v \partial_x g_1 \geq  \left( M(v)\rho_{g_1} - g_1 \right) +r  \rho_{g_1} \left( M(v)-g_1 \right)\,,\smallskip\\
\partial_t g_2 +  v\partial_x g_2 \leq  \left( M(v)\rho_{g_2} - g_2 \right) +r  \rho_{g_2} \left( M(v)-g_2 \right) \,,\\
\end{array}
\end{equation*}
in the sense of distributions. Assume in addition that $g_2$ satisfies $g_2(t,x,v)\leq M(v)$ for all $(t,x,v)\in \R_+\times \R\times V$.  Then $g_2(t,x,v)\leq g_1(t,x,v)$ for all $(t,x,v)\in\R_+\times \R\times V$.

Assume in addition that $V$ is an interval, and  that $\inf_V M>0$. If there exists $(x_0,v_0)$ such that $g_2 (0,x_0,v_0)> g_1 (0,x_0,v_0)$, 
then one has $g_1 (t,x,v) > g_2 (t,x,v)$ for all $(t,x,v) \in \R_+ \times \R\times V$ such that $|x-x_0| < v_{\rm max} t$. 
\end{prop}

\begin{rmq}
 If $V=\R$, then this statement reads as in the parabolic framework: if $g_2\geq g_1$ and $g_2\not\equiv g_1$ at $t=0$, then $g_2 >g_1$ for all $t>0$. In the case $V = [-v_{\rm max},v_{\rm max}]$ we have to take into account finite speed of propagation, obviously.  
\end{rmq}

\section{Existence and construction of travelling wave solutions}\label{existence}

We assume throughout this Section that $V=\supp M$  is compact. 
We construct the travelling waves for $c\in [c^*,v_{\rm max})$. 
The proof is divided into several steps. 
It is based on a sub and supersolutions method.


\subsection{The linearized problem.}

The aim of this first step is to solve the linearized equation of \eqref{eqkinwave} at $+ \infty$, in the regime of low density $f\ll 1$. 
Such an achievement gives information about the speed and the space decreasing rate of a travelling wave solution of the nonlinear problem, 
as for the Fisher-KPP equation. \eb{The linearization of \eqref{eqkinwave} at $f=0$} writes 
\begin{equation} \label{lineq:kinKPP2}
( v - c  )\partial_{x } f = \left( M(v) \rho_f - f \right) + r M(v) \rho_f \, ,
\end{equation}
We seek a solution having exponential decay at $+\infty$. More specifically we separate the variables in our ansatz: $f(x,v) = e^{-\lambda x} F(v)$, 
with $\int_V F(v)dv=1$. The next Proposition gathers the results concerning the linear problem.

\begin{prop}[Existence of a minimal speed for the linearized equation]\label{propdisp} 
There exists a minimal speed $c^*$ such that for all $c \in \left[c^*,v_{\rm max} \right)$, 
there exists $\lambda> 0$ such that $f_{\lambda}(x,v) = e^{-\lambda x} F_{\lambda}(v)$ is a nonnegative solution of \eqref{lineq:kinKPP2}. The profile $F_{\lambda}$ is explicitely given by
\begin{equation*}
F_{\lambda} (v)= \frac{(1+r)M(v)}{1 + \lambda ( c - v )} \geq 0\,.
\end{equation*}
The admissible $(\lambda,c)$ are solutions of the following dispersion relation,
\begin{equation}\label{eq:dispersion}
 \int_V \frac{(1 + r)M(v)}{1 +\lambda(c - v)}\, dv = 1\, .
\end{equation}
Moreover, among all possible $\lambda$ for a given $c$, the minimal one $\lambda_c$ is well defined and isolated.
\end{prop}

\begin{rmq}
Here appears the crucial assumption  on the boundedness of $V$. If this condition is not fulfilled, it is never possible to ensure that the profile $F_{\lambda}$ is nonnegative since the denominator is linear with respect to $v$.
\end{rmq}

\begin{proof}[\bf Proof of Proposition \ref{propdisp}.] 
\textbf{\# Step 1.} 
Plugging the ansatz $f_{\lambda}(x,v) = e^{-\lambda x}F_{\lambda}(v)$  into \eqref{lineq:kinKPP2}  yields
\begin{equation} \label{eq:TW}
(c-v)\lambda F_{\lambda}(v) = \left( M(v)  - F_{\lambda}(v) \right) + r M(v) \, .
\end{equation}
The profile is given by
\begin{equation*}
F_\lambda (v)= \frac{(1+r)M(v)}{1 + \lambda  ( c - v )}.
\end{equation*}
The dispersion relation reads $\int_V F_\lambda(v)dv=1$, or equivalently
\eqref{eq:dispersion}.
Moreover, we require the profile $F_\lambda$ to be nonnegative, which gives the condition $1+\lambda (c-v)>0 $ for all $v\in V$, which implies $\lambda < \frac{1}{v_{\rm max}-c}$. 

From now on, we focus on the existence of solutions $(\lambda,c)$ of \eqref{eq:dispersion}, with $c\in [0,v_{\rm max})$ and $\lambda \in \left[ 0, \frac{1}{v_{\rm max}-c}\right)$. Let us denote 
\begin{equation} \label{eq-defI} 
I(\lambda;c) =  \int_V \frac{(1 + r)M(v)}{1 + \lambda(c - v)}\, dv.
\end{equation}
so that we look for solutions of $I(\lambda;c)= 1$. 

\medskip

\noindent{\bf \# Step 2.} Technically speaking, for all $c\in [0,v_{\rm max})$, the function $ \lambda \mapsto I(\lambda;c)$ is analytic over $\left[ 0, \frac{1}{v_{\rm max}-c}\right)$
Indeed, as $v\mapsto v^n M(v)$ is integrable for all $n$, it is clear that 
\begin{equation*}
I(\lambda;c)=\sum_{n\geq 0} (1+r)\lambda^n\int_V M(v) (v-c)^n dv
\end{equation*}
is the analytic development of $I$ for  $\lambda \in\left[ 0, \frac{1}{v_{\rm max}-c}\right)$. 
Next we observe that $c\mapsto I(\lambda;c)$ is decreasing for all $\lambda \in \left( 0, \frac{1}{v_{\rm max}-c}\right)$, and that $\lambda\mapsto I(\lambda;c)$ is strictly convex.
Moreover, the function $I$ satisfies the following properties:
\begin{align*}
&I(0;c) = 1+ r>1\,,\\
& I(\lambda;0) = (1 + r)\int_V \dfrac{M(v)}{1 - \lambda v}\, dv>1 , \quad \hbox{for all }   \lambda \in\left[ 0, \frac{1}{v_{\rm max}}\right)\\
& I \left(\lambda; v_{\rm max}\right) = (1 + r) \int_V \dfrac{M(v)}{1 + \lambda(v_{\rm max} - v)}\, dv \, \xrightarrow[\lambda\to +\infty]{} 0 \, .
\end{align*}
The last property relies on the Lebesgue's dominated convergence theorem since  $M\in L^1(V)$.

\medskip

\noindent{\bf \# Step 3.} Assume first that $ \frac{M(v)}{v_{\rm max}-v} \not\in L^1(V)$. Then Fatou's lemma gives
$$\liminf_{\lambda \nearrow \frac{1}{v_{\rm max}- c} } I(\lambda;c) = \liminf_{\lambda \nearrow \frac{1}{v_{\rm max}- c} } \int_V \frac{M(v)}{ 1 + \lambda (c-v)}dv \geq 
\int_V \liminf_{\lambda \nearrow \frac{1}{v_{\rm max}- c} } \frac{M(v)}{ 1 + \lambda (c-v)}dv =\int_V \frac{M(v)}{ 1 - \frac{v-c}{v_{\rm max}- c}}dv=+\infty.$$
As a consequence, 
$\theta(c)=\min\left\{ I(\lambda;c):  \lambda\in \left[0, \frac{1}{v_{\rm max}- c}\right)\right\}$ is well defined and finite for all $c\in [0,v_{\rm max})$. 
It follows from the earlier properties that $\theta(0)>1$ and $\theta(v_{\rm max})= 0$. Moreover, the regularity and monotonicity 
properties of $I$ guarantee that $\theta$ is continuous and decreasing. 
Hence, there exists $c^*$ such that $\theta(c^*)=1$ and there exists $\lambda_{c^*}$ such that $I(\lambda_{c^*};c^*)=1$. 

Next, for all $c\in (c^*,v_{\rm max})$, as $c\mapsto I(\lambda;c)$  is decreasing, one has 
$I(\lambda_{c^*};c)<1$ for all $c>c^*$. Thus, as $I(0;c)>1$, there exists $\lambda$ such that $I(\lambda;c)=1$ for all $c>c^*$.

Second, consider a general $M\in \mathcal{C}^0(V)$ possibly vanishing at $v = v_{\max}$. To recover the first step, we define for $n \in \N^*$ a new distribution $M_n = \frac{M+1/n}{1+|V|/n}$ over $V$ (and $0$ outside of $V$), where $|V|$ is the  measure of $V$. 
Then $ \frac{M_n(v)}{v_{\rm max}-v} \not\in L^1(V)$ since $M_n (v_{\rm max}) \geq \frac{1/n}{1+|V|/n}>0$, and thus the earlier step yields that there exists a sequence $c_n^*$ of minimal speeds
associated with $(M_n)_n$. We also associate $I_n$ with $M_n$ through (\ref{eq-defI}). We define 
\begin{equation*}
c^*= \limsup_{n\to \infty} c_n^*,
\end{equation*}
and we now show that it is the minimal speed.


\begin{itemize}
\item Take $c<c^*$. Then for all $\lambda\in \left(0, \frac{1}{v_{\rm max}-c}\right)$ and for some arbitrarily large $n$ so that $\lambda\in \left(0, \frac{1}{v_{\rm max}- c^*_n}\right)$, one has 
\begin{equation*} 
I_n(\lambda; c) = I_n(\lambda; c^*_n)-\int_c^{c^*_n}\partial_c I_n (\lambda,c')dc' \geq 1-\int_c^{c^*_n}\partial_c I_n (\lambda,c')dc'  \geq
1+\frac{(1 + r) \lambda }{\left(1 + \lambda( c^*_n + v_{\rm max})\right)^2}(c^*_n-c).
\end{equation*}
Because $I_n(\lambda;c) \underset{n \to + \infty}{\to} I(\lambda;c)$ as $n\to +\infty$, we get 
\begin{equation*}
I(\lambda;c)\geq 1 +\frac{(1 + r) \lambda}{1 + \lambda( c^* + v_{\rm max})}(c^*-c)>1.
\end{equation*}
Thus $I(\lambda;c)=1$ has no solution for $\lambda\in \left(0, \frac{1}{v_{\rm max}- c}\right)$ if $c<c^*$. 

\item Assume that $c> c^*$. Then one has $c>c_n^*$ when $n$ is large enough and thus for all $n$ sufficiently large, there exists 
$\lambda_n\in \left(0, \frac{1}{v_{\rm max}- c}\right)$ such that $I_n (\lambda_n;c)=1$. Up to extraction, one may assume that $(\lambda_n)_n$ converges to some 
$\lambda_\infty\in \left[0, \frac{1}{v_{\rm max}- c}\right]$. Fatou's lemma yields $I(\lambda_\infty;c)\leq 1$. Hence, there exists a solution $\lambda \in \left[0, \frac{1}{v_{\rm max}- c}\right]$
of $I(\lambda;c)=1$ and obviously $\lambda\neq 0$ since $I(0;c)>1$. 
\item Lastly, if $c=c^*$, we know that for all $k\in \N^*$, there exists $\lambda_k\in \left(0, \frac{1}{v_{\rm max}- (c^*+1/k))}\right]$ such that $I(\lambda_k; c^*+1/k)=1$. 
Assuming that $\lambda_k\to \lambda\in \left[ 0,\frac{1}{v_{\rm max}- c} \right]$ as $k\to +\infty$, we get $I(\lambda;c^*)=1$.
\end{itemize}
\end{proof}

\begin{lem}[Spatial decay rate]\label{minlambda}
For all $c\in \left[c^*,v_{\rm max} \right)$, the quantity 
\begin{equation*}
\lambda_c = \min \{ f>0: I( f;c)=1\}.
\end{equation*}
is well-defined. Moreover, for all $c \in \left( c^* , v_{\rm max} \right)$, if $\gamma>0$ is small enough, then $I(\lambda_c+\gamma;c)<1$.
\end{lem}
%
%

\begin{proof}[\bf Proof of Lemma \ref{minlambda}.]
We know from the definition of $c^*$ that for all $c\in [c^*,v_{\rm max})$, the set $ \Lambda_c = \left \lbrace  f>0: I( f;c)=1 \right \rbrace$ is not empty. 
Thus, we can take a minimizing sequence $\lambda_n$ which converges towards the infimum of $\Lambda_c$. As this sequence is bounded, one can assume, up to extraction, 
that $\lambda_n\to\lambda_c \geq 0$. Then  Lebesgue's dominated convergence theorem gives $I(\lambda_c;c)=1$. Hence 
$\lambda_c = \min \Lambda_c$. 

Next, we have already noticed in the proof of Proposition \ref{propdisp} that $I(\lambda_{c^*},c)<1$ for all $c>c^*$. As 
$I(0,c)=1+r>1$, the definition of $\lambda_c$ yields $\lambda_c <\lambda_{c^*}$. The conclusion follows from the strict convexity of 
the function $\lambda\mapsto I(\lambda;c)$. 
\end{proof}


\subsection{Construction of sub and supersolutions when $c\in (c^*, v_{\rm max})$.}

In this step we construct sub and supersolutions for \eqref{eq:kinKPP}. We fix $c\in \left( c^*,v_{\rm max}  \right)$ and we 
denote $\lambda=\lambda_c$ \eb{for legibility}.

\begin{lem}[Supersolution] \label{lem-super}
Let 
\begin{equation*}
\overline{f}(x ,v)= \min \left \lbrace M(v),e^{-\lambda x }F_\lambda (v) \right \rbrace.
\end{equation*} 
Then $\overline{f}$ is a supersolution of \eqref{eqkinwave}, that is, it satisfies in the sense of distributions:
\begin{equation} \label{eq-super} 
 (v-c) \partial_{x } \overline{f} \geq \left( M(v)\rho_{\overline{f}} - \overline{f} \right) +r \rho_{\overline{f}} \left( M(v)-\overline{f} \right), \quad (x ,v) \in \R \times V.
\end{equation}
\end{lem}

\begin{lem}[Subsolution] \label{lem-sub}
There exist $A>0$ and $\gamma>0$ such that if 
\begin{equation*}
\underline{f}(x ,v)= \max \left\{0,e^{-\lambda x }F_\lambda (v)-A e^{-(\lambda+\gamma) x }F_{\lambda +\gamma}(v)  \right\}\,,
\end{equation*} 
then $\underline{f}$ is a subsolution of \eqref{eqkinwave}, that is satisfies in the sense of distributions:
\begin{equation} \label{eq-subper}
 (v-c) \partial_{x } \underline{f} \leq \left(M(v)\rho_{\underline{f}} - \underline{f}\right) + r\rho_{\underline{f}} \left( M(v)-\underline{f} \right), \quad (x ,v) \in \R \times V.
\end{equation}
\end{lem}


\begin{proof}[\bf  Proof of Lemma \ref{lem-super}.]
First, $(x ,v)\mapsto e^{-\lambda x }F_\lambda (v)$ and $(x ,v)\mapsto M(v)$ both clearly satisfy \eqref{eq-super} since $\overline{f}\geq 0$. 
Next, as $\overline{f}$ is continuous, it immediately follows from the jump formula that, as a minimum of two supersolutions, it is a supersolution
of (\ref{eq-super}) in the sense of distributions. 
\end{proof}

\begin{proof}[\bf  Proof of Lemma \ref{lem-sub}.] 
The same arguments as in the proof of Lemma \ref{lem-super} yield that it is enough to prove that (\ref{eq-subper}) is satisfied by $\underline{f}$ over the open set $\{\underline{f}>0\}$. 
As $c>c^*$, Proposition \ref{propdisp} gives $\gamma\in (0,\lambda)$ small enough such that $I(\lambda+\gamma;c)<1$ and $F_{\lambda + \gamma}(v) > 0$.
We compute the linear part:
\begin{equation*}
 (v-c) \partial_{x } \underline{f} - \left( M(v) \rho_{\underline{f}} - \underline{f} \right) - r\rho_{\underline{f}} M(v)\\
=A \left(I(\lambda+\gamma,c)-1\right)(1 + r)e^{-(\lambda+\gamma)x }M(v).
\end{equation*}
To prove the Lemma, we now have to choose a relevant $A$ such that 
\begin{equation}\label{rhs}
r  \underline{f}\rho_{\underline{f}} \leq A (1 + r) M(v)\left( 1-I(\lambda+\gamma,c)\right) e^{-(\lambda+\gamma) x }\, .
\end{equation}
holds for all $( x  , v ) \in \R \times V$. As $\underline{f}(x ,v)\leq e^{-\lambda x }F_\lambda (v)$ for all $(x ,v)\in\R\times V$, one has $\rho_{\underline{f}} (x ) \leq e^{-\lambda x }$
and thus it is enough to choose $A$  such that
\begin{align} 
& re^{- 2 \lambda x } F_{\lambda}(v)  \leq A (1 + r) M(v)\left( 1-I(\lambda+\gamma,c)\right) e^{-(\lambda+\gamma) x }\,, \nonumber\\
& \dfrac{r e^{-(\lambda - \gamma)x}}{1-I(\lambda+\gamma,c)}\left( \dfrac1{1 + \lambda (c-v)}\right) \leq A\, . \label{eq:estimate A below}
\end{align}
On the other hand for all $(x ,v)\in\R\times V$ such that $\underline{f}(x ,v)>0$, we have $F_\lambda (v) > A e^{-\gamma x } F_{\lambda +\gamma} (v)$, meaning that 
\begin{equation*}
e^{- \gamma x }< \frac1 A \left(\frac{1 + (\lambda+\gamma) ( c-v)}{1 +\lambda  ( c-v)}\right)\,.
\end{equation*}
Plugging this estimate into \eqref{eq:estimate A below}, it is enough to choose $A$ such that 
\begin{align*}
\left( \frac1 A \left(\frac{1 + (\lambda+\gamma) ( c-v)}{1 +\lambda  ( c-v)}\right) \right)^{\frac{\lambda-\gamma}\gamma} \dfrac{r }{1-I(\lambda+\gamma,c)}\left( \dfrac1{1 + \lambda (c-v)}\right) & \leq A \\
\sup_{v\in V} \left\{   \left(\frac{1 + (\lambda+\gamma) ( c-v)}{1 +\lambda  ( c-v)}\right)^{\frac{\lambda-\gamma}\gamma} \dfrac{r }{1-I(\lambda+\gamma,c)}\left( \dfrac1{1 + \lambda (c-v)}\right) \right\} & \leq A^{\frac\lambda\gamma}\, .
\end{align*}
This concludes the proof since such a $A$ obviously exists.
\end{proof}


\subsection{Construction of the travelling waves in the regime $c\in (c^* , v_{\rm max})$.}

Let $c\in (c^* , v_{\rm max} )$, where $c^*$ denotes the minimal speed of Proposition \ref{propdisp}. 
In order to prove the existence of travelling waves, we will prove that the solution of the following evolution equation,
corresponding to equation \eqref{eq:kinKPP} in the moving frame at speed $c$, converges to a travelling wave as $t\to +\infty$:
\begin{equation}\label{eq-approx}\left\{ \begin{array}{l}
 \partial_t g+  (v-c) \partial_x g =   M(v)\rho_g - g   +r \rho_g \left( M(v)-g\right) \hbox{ in } \R\times V,\medskip\\
g(0,x,v)=\overline{f}(x,v) \hbox{ for all } (x,v) \in\R\times V.\\ \end{array} \right.
\end{equation}
The well-posedness of equation \eqref{eq-approx} immediately follows from Proposition \ref{Cauchypbm}.
Let now derive some properties of the function $g$ from Proposition \ref{prop-mp}.

\begin{lem}\label{lem-sandwich}
For all $(t,x,v)\in \R_+\times\R\times V$, one has $\underline{f}(x,v)\leq g(t,x,v)\leq \overline{f}(x,v).$
\end{lem}

\begin{proof}[\bf  Proof of Lemma \ref{lem-sandwich}.] 
As $\underline{f}$ is a subsolution of (\ref{eq-approx}) and $\overline{f}$ is a supersolution of (\ref{eq-approx}), 
with $\underline{f} (x,v)\leq \overline{f} (x,v)$ for all $(x,v) \in\R\times V$, this result is an immediate corollary of Proposition \ref{prop-mp}. 
\end{proof}

\begin{lem}\label{lem-monotonicity}
For all $(t,v)\in \R_+\times V$, the function $x\in\R \mapsto g(t,x,v)$ is nonincreasing.  
\end{lem}

\begin{proof}[\bf Proof of Lemma \ref{lem-monotonicity}.] 
Take $h\geq 0$ and define $g_h(t,x,v)=g(t,x+h,v)$. Then as $\overline{f}$ is nonincreasing in $x$, one has $g_h(0,x,v)\leq g(0,x,v)$ for all 
$(x,v)\in\R\times V$. Proposition \ref{prop-mp} yields that $g_h(t,x,v)\leq g(t,x,v)$ for all $(t,x,v)\in \R_+\times \R \times V$. 
\end{proof}

\begin{lem} \label{lem-timedec}
For all $(x,v)\in \R\times V$, the function $t\in\R_+ \mapsto g(t,x,v)$ is nonincreasing.  
\end{lem}

\begin{proof}[\bf Proof of Lemma \ref{lem-timedec}.] 
Take $\tau\geq 0$ and define $g_\tau(t,x,v)=g(t+\tau,x,v)$. Then Lemma \ref{lem-sandwich} yields that 
$g_\tau(0,x,v)\leq \overline{f}(x,v)=g(0,x,v)$ for all $(x,v)\in\R\times V$. Hence, Proposition \ref{prop-mp} gives 
$g_\tau(t,x,v)\leq g(t,x,v)$ for all $(t,x,v)\in\R_+\times \R\times V$. 
\end{proof}

\begin{lem}\label{lem-lipsch}
The family $\left( g(t,\cdot,\cdot)\right)_{t\geq 0}$ is uniformly continuous with respect to $(x,v) \in \R \times V$. Moreover, for any $A\in (c^*,v_{\rm max})$,
the continuity constants does not depend on $c\in (c^*,A)$. 
\end{lem}

\begin{proof}[\bf Proof of Lemma \ref{lem-lipsch}.] We begin with the space regularity. 
Let $|h|<1$. The function $g(0,x,v)=\overline{f}(x,v)=\min \{ M(v) , e^{-\lambda x} F_\lambda (v)\}$ is such that $\log g(0,x,v)$ is Lipschitz-continuous with respect to $x$. Therefore there exists a constant $C_0>0$ such that for all $(x,v) \in \R\times V$, we have
$g(0,x+h,v) \leq (1+C_0 |h|) g (0,x,v).$ 
As $1+C_0 |h|>1$, it is easily checked that $(t,x,v) \mapsto (1+C_0 |h|) g (t,x-h,v)$ is a supersolution of \eqref{eq-approx}. 
Hence Proposition \ref{prop-mp} yields that 
\[g(t,x,v) \leq  (1+C_0 |h|) g (t,x-h,v) \quad \hbox{ for all } (t,x,v) \in \R_+ \times \R\times V\, .\]
Hence the function $\log g$ is Lipschitz continuous with respect to $x$. Since the function $\log g$ is bounded from above, $g = \exp(\log g)$ is also Lipschitz continuous with respect to $x$. The Lipschitz constant is uniform with respect to $c \in (c^*,A)$ and 
$\lambda \in \left(0,1/(v_{\rm max}-c)\right)$.

We now come to the velocity regularity. For the sake of clarity we first consider the case where $M$ is $\mathcal C^1$ on $V$. The function $v\mapsto g(0,x,v)$ is  $\mathcal C^1$ too. We introduce $g_v = \partial_v g$. It satisfies the following equation
\[  \partial_t g_v+  (v-c) \partial_x g_v + (1 + r \rho_g )g_v =   (1 + r   ) M'(v)\rho_g  - \partial_x g \hbox{ in } \R\times V\, . \]
Multiplying the equation by $\sign g_v$ we obtain
\[  \partial_t |g_v|+  (v-c) \partial_x |g_v| + (1 + r \rho_g )|g_v| \leq   (1 + r   ) |M'(v)|\rho_g  + |\partial_x g| \hbox{ in } \R\times V\, . \]
The l.h.s. is linear with respect to $|g_v|$ and satisfies the maximum principle. The r.h.s. is uniformly bounded since $0\leq \rho_g \leq 1$ and $g$ is uniformly Lipschitz with respect to $x$. Obviously the constant $(1+r) \sup_V |M'(v)| + \sup_{\R_+\times\R\times V}|\partial_x g|$ is a supersolution. 
We deduce that $g_v$ is uniformly bounded over $\R_+\times\R\times V$. 

In the case where $M$ is only continuous over the compact set $V$, thus uniformly continuous, we shall use the method of translations again. However we have to be careful since $V$ is bounded. Let $0<h<1$. We introduce $H(v) = \max(v + h,v_{\rm max}) - v$. The function $g_H(t,x,v) = g(t,x,v + H(v)) - g(t,x,v)$ satisfies the following equation
\[  \partial_t g_H + (v-c)\partial_x g_H + (1 + r\rho_g) g_H = (1 + r)(M(v + H(v)) - M(v))\rho_g - H(v) \partial_x g (t,x,v+H(v))\, .  \]
Let  $\eps>0$. There exists $\delta >0$ such that for $0<h<\delta$ we have $|g_H(0,x,v)|\leq \delta$ and $|M(v + H(v)) - M(v)|<\delta$. Moreover we have obviously $0<H(v)<\delta$. We conclude using the same argument as in the $\mathcal C^1$ case. The modulus of uniform continuity is uniform with respect to $c \in (c^*,A)$ and 
$\lambda \in \left(0,1/(v_{\rm max}-c)\right)$. 
\end{proof}

We are now in position to prove the existence of travelling waves of speed $c$, except for the minimal speed~$c^*$. 

\begin{proof}[\bf Proof of the existence in Theorem \ref{existence-tw} when $c>c^*$.]

Gathering Lemmas \ref{lem-sandwich}, \ref{lem-monotonicity} and \ref{lem-timedec}, we know that 
\begin{equation*}
f(x,v)=\lim_{t\to +\infty} g(t,x,v),
\end{equation*} 
is well-defined for all $(x,v)\in\R\times V$, that $f(\cdot,v)$ is nonincreasing in $x$ for all $v$ and that $\underline{f}\leq f\leq \overline{f}$. 

Let now prove that $f$ defines a travelling wave solution of \eqref{eq:kinKPP2}. As $g$ satisfies \eqref{eq-approx}, converges pointwise 
and is bounded by the 
locally integrable function $\overline{f}$, it  follows from Lebesgue's dominated convergence theorem that $f$ satisfies \eqref{eqkinwave} in the sense of 
distributions. Moreover, Lemma \ref{lem-lipsch} ensures that $f$ is continuous with respect to $(x,v)$. 

We next check the limits towards infinity. Let $f^\pm(v)= \lim_{x\to \pm \infty} f(x,v)$. Thanks to $f\leq \overline{f}$, the Lebesgue dominated convergence theorem gives 
$\rho_{f^\pm}= \int_V f^\pm(v)dv\leq 1$. On the other hand, we get 
\begin{equation} \label{eq-limit} 
\left( M(v)\rho_{f^\pm} - f^\pm (v) \right) + r\rho_{f^\pm} \left( M(v) - f^\pm(v) \right) = 0 
\end{equation}
in the sense of distributions. 
Integrating \eqref{eq-limit} over the compact set $V$, we deduce that $\rho_{f^\pm} \left( 1 - \rho_{f^\pm} \right) = 0$, \textit{i.e.} 
that $\rho_{f^\pm}= 0$ or $1$. As $f$ is nonincreasing and $\underline{f}\leq f\leq \overline{f}$,
one necessarily has $\rho_{f^+}=0$ and $\rho_{f^-}=1$. Finally, (\ref{eq-limit}) gives $f^+(v)=0$ and $f^-(v)=M(v)$ for all $v\in V$. 
This gives the appropriate limits. 
\end{proof}


\subsection{Construction of the travelling waves with minimal speed $c^*$.}

\begin{proof}[\bf Proof of the existence in Theorem \ref{existence-tw} when $c=c^*$.]
Consider a decreasing sequence $(c_n)$ converging towards $c^*$. We already know that for all $n$, equation \eqref{eq:kinKPP} admits a travelling wave solution 
$u_n (t,x,v) = f_n (x-c_n t ,v)$, with $f_n (-\infty,v)=M(v)$ and $f_n (+\infty,v)=0$, and $ z \mapsto f_n ( z ,v)$ is nonincreasing. Up to translation, we can assume that 
$\rho_{f_n}(0) = 1/2$. 
Moreover, Lemma \ref{lem-lipsch} ensures that the functions $(f_n)_n$ are uniformly continuous with respect to $(x,v) \in \R\times V$ since the continuity stated in Lemma \ref{lem-lipsch}
is uniform with respect to $c\in (c^*,A)$ for any $A\in (c^*,v_{\rm max})$. Thanks to the Ascoli theorem and a diagonal extraction process, 
we can assume that the sequence $(f_n)_n$ converges locally uniformly in $(x,v) \in \R\times V$ to a function $f$. 
Clearly $f$ satisfies \eqref{eqkinwave} in the sense of distributions. Moreover, as $f$ is nonincreasing with respect to $x$, one could recover the appropriate limits at 
infinity with the same arguments as in the proof of the existence of travelling waves with speeds $c>c^*$. 
\end{proof}


\subsection{Non-existence of travelling wave solutions in the subcritical regime $c\in [0,c^*)$.}


\begin{lem} \label{lem:posHarnack} Assume that $\inf_V M(v)>0$. For all $0\leq c< c^*$ there exists $c<c_0<c^*$ and a nonnegative, arbitrarily small, compactly supported function $\sub(x,v)$ which is a subsolution of 
\begin{equation} \label{eq:kinKPPc}
(v-c^0) \partial_x f  = M(v) \rho_f - f + r\rho_f \left( M(v) -f\right) \quad \hbox{ in } \R\times V\,.
\end{equation}
\end{lem}

\begin{proof}[\bf Proof of Lemma \ref{lem:posHarnack}]
For the sake of clarity we emphasize the dependence of the function $I$ \eqref{eq-defI} upon the growth rate $r>0$:
$$I_r(\lambda;c) = \int_V \frac{(1 + r) M(v)}{1 + \lambda(c - v)}\, dv\, .$$
We denote by $c^*_{r} $ the smallest speed such that there exists a solution $\lambda>0$ of 
$I_r (\lambda,c)=1$ (see Proposition \ref{propdisp}). 

Let $\delta>0$. By continuity we can choose $\delta$ so small that $c<c^*_{r-\delta}$. We claim that there exists $(c^0,\lambda^0)$ such that $I_{r-\delta}(\lambda^0;c^0) = 1$, with $c<c^0<c^*_{r-\delta}$ and $\lambda^0\in \C\setminus \R$. Indeed we know from the proof of Proposition \ref{propdisp} [Step 3] that $\lambda^*_{r} < 1/(v_{\rm max} - c^*_{r})$ under the assumption $v\mapsto M(v) / (v_{\rm max}-v) \notin L^1 (V)$. Using a continuity argument we also have the strict inequality $\lambda^*_{r-\delta} < 1/(v_{\rm max} - c^*_{r-\delta})$, uniformly with respect to $\delta$. The complex function $\lambda\mapsto I_{r-\delta}(\lambda;c^*_{r-\delta})$ is analytic in a neighborhood of $\lambda^*_{r-\delta}$. Hence, the Rouch\'e theorem yields that there exists $c^0<c^*_{r-\delta}$ such that the equation $I_{r-\delta}(\lambda;c^0) = 1$ has a solution $\lambda^0\in \C$ with $\lambda^0$ arbitrarily close to $\lambda^*_{r-\delta}$. We denote by $F^0(v)$ the corresponding velocity profile,
\[ F^0 (v)= \frac{(1+r-\delta)M(v)}{1 + \lambda^0  ( c^0 - v )}\,,\quad \int_V F^0(v)\, dv = 1\,.
 \]
By continuity we can choose $c^0$ and $\lambda^0$ such that 
$\text{Re}\left(F^0(v)\right)>0$ holds for all $v\in V$. By the very definition of $c^*_{r-\delta}$, we have $\lambda^0\notin \R$. We denote $\lambda^0 = \lambda_{R} +i \lambda_{I} $. Recall that we have  the strict inequality $\lambda^*_{r-\delta} < 1/(v_{\rm max} - c^*_{r-\delta})$, uniformly with respect to $\delta$. Using a continuity argument we can   impose that $\lambda_R<1/(v_{\rm max}-c^0)$. 
Now define the real function $\sub^0$ by 
\begin{equation}\label{reecriturew} 
h^{0}(x,v) = \text{Re} \left(e^{-\lambda^0 x}F^0 (v)\right) = e^{-\lambda_{R}x}\left[ \text{Re} \left(F^0(v)\right) \cos (\lambda_{I}x)+ \text{Im} \left(F^0(v)\right) \sin (\lambda_{I}x) \right],
\end{equation}
One has $\sub^0(0,v)>0$ and $\sub^0(\pm \pi/\lambda_{I},v)<0$ for all $v\in V$. 
Thus, there exists an interval $[b_{1},b_{2}]\subset \R$ and a bounded domain $D\subset  [b_{1},b_{2}]\times V$  such that:
\begin{equation*} \left\{ \begin{array}{l}
\sub^0(x,v)>0 \quad  \hbox{for all} \ (x,v)\in D,\medskip\\
\sub^0(x,v)=0 \quad  \ \hbox{for} \ (x,v)\in\partial D.\\
\end{array} \right. 
\end{equation*}
On the other hand, as $\lambda_R<1/(v_{\rm max}-c^0)$, there exists a constant $C(\delta)$ such that
\begin{equation*}
(\forall v\in V)\quad |\sub^0 (x,v)| \leq  e^{-\lambda_R b_1} |F^0 (v)|=  e^{-\lambda_R b_1}\frac{(1+r-\delta)M(v)}{|1+\lambda^0 (c^0-v)|}\leq 
C(\delta)  M(v)\,.
\end{equation*}
Hence, one can choose $\overline{\kappa}>0$ small enough such that 
\begin{equation*}
r\overline{\kappa} \sub^0 (x,v) \leq \frac\delta2 M(v) \quad \hbox{for all} \; (x,v) \in\R\times V\,.
\end{equation*}
For all $\kappa \in (0,\overline{\kappa})$ we deduce from $I_{r-\delta}(\lambda^0;c^0)=1$ the following identities,
\begin{align}
\kappa (v-c^0) \partial_x\left( e^{-\lambda^0 x}F^0 (v)\right) + \kappa\left( e^{-\lambda^0 x}F^0 (v)\right)  & = \kappa e^{-\lambda^0 x}(1+ r-\delta)M(v) \nonumber\\
& = \kappa (1+ r-\delta) M(v) \int_V e^{-\lambda^0 x}F^0(v') \, dv' \, .\nonumber 
\end{align}
Taking real part on both sides, we get for $(x,v)\in D$,
\begin{align*}
(v-c)  \partial_x \left(\kappa \sub^0\right) + \kappa \sub^0 & = (1+r-\delta) M(v) \int_V \kappa \sub^0(x,v')\, dv' 
\\  
& = M(v) \rho_{\kappa \sub^0} + rM(v) \rho_{\kappa \sub^0}   - \delta M(v) \rho_{\kappa \sub^0}
\\ 
& \leq M(v) \rho_{\kappa \sub^0} + r\left(M(v) - \kappa \sub^0\right)  \rho_{\kappa \sub^0} \, .
\end{align*}
Hence $\kappa \sub^0$ is a subsolution of \eqref{eq:kinKPPc} for all $\kappa \in (0,\overline{\kappa})$ on $D$. We deduce that the truncated function $\sub = \max(0, \kappa \sub^0)$ is a subsolution of \eqref{eq:kinKPPc} over $\R\times V$. 
\end{proof}

\begin{proof}[\bf Proof of the non-existence in Theorem \ref{existence-tw}.]
Assume that $f\in \mathcal{C}^0 (\R\times V)$ is a travelling wave solution of \eqref{eqkinwave} of speed $c\in (0,c^*)$. 
%
%
According to Lemma \ref{lem:posHarnack}, there exists $c<c^0<c^*$ and a nonnegative compactly supported subsolution $\sub$ of \eqref{eqkinwave} 
with speed $c^0$. 
As $f$ is positive and continuous, we can decrease $\sub$ so as to obtain $f\geq \sub$. Let $g_1(t,x,v)= f(x-ct,v)$ and $g_2 (t,x,v)= \sub(x-c^0t,v)$. These two functions 
are respectively a solution and a subsolution of \eqref{eq:kinKPP}. As $g_1 (0,x,v) = f(x,v)\geq \sub(x,v) = g_2 (0,x,v)$ for all $(x,v) \in \R\times V$, 
Proposition \ref{prop-mp} implies 
$$g_1 (t,x,v) = f(x-ct,v)\geq \sub(x-c^0t,v) = g_2 (t,x,v) \quad \hbox{ for all } (t,x,v) \in \R_+\times \R\times V.$$
Taking $x=c^0t$ and letting $t\to +\infty$, we get 
$$0=\lim_{t\to +\infty} f\left( (c^0-c)t,v \right) \geq \sub(0,v)\,.$$
This is a contradiction. 
\end{proof}


\subsection{Proof of the spreading properties}

\begin{proof}[\bf Proof of Proposition \ref{prop:spreadingbounded}]
 1. Let $c>c^*$. Consider first the initial datum
$$\widetilde{g}^0 (x,v) = \left\{ \begin{array}{lll} M(v)  &\hbox{ if }& x<x_R\,,\\
                       0 &\hbox{ if }& x\geq x_R\, ,\\ 
                      \end{array}\right. $$
and let $\widetilde{g}$ the solution of the Cauchy problem \eqref{Cauchypbm}. Denote by $f$ a travelling wave of minimal speed $c^*$. 
There exists $\kappa>1$ such that $\widetilde{g}^0 (x,v) \leq \kappa f(x,v)$ for all 
$(x,v) \in \R\times V$. It is straightforward to check that $g_1 (t,x,v) = \kappa f(x-c^*t,v)$ is a supersolution of \eqref{eq:kinKPP}. 
Hence, the comparison principle of Proposition \ref{prop-mp} implies that $\widetilde{g}(t,x,v) \leq g_1 (t,x,v)$ for all $(t,x,v) \in \R_+ \times \R\times V$.
In particular we have, 
$$\widetilde{g}(t,ct,v) \leq g_1 (t,ct,v) = \kappa f\left( (c-c^*) t,v\right) \quad \hbox{ for all } (t,v) \in \R_+ \times V.$$
As $f(+\infty,v) = 0$ for all $v\in V$ and $c>c^*$, we get $\lim_{t\to +\infty} \widetilde{g}(t,ct,v) = 0$ for all $v\in V$. 

On the other hand, as $\widetilde{g}^0$ is nonincreasing with respect to $x\in\R$ it  follows from the comparison principle that $x\mapsto \widetilde{g}(t,x,v)$ is nonincreasing (see Lemma \ref{lem-monotonicity}). Thus $\widetilde{g}(t,x,v) \leq \widetilde{g}(t,ct,v)$ for all $x\geq ct$ and the conclusion follows. 

For a general initial datum $g^0$ satisfying the hypotheses of Proposition \ref{prop:spreadingbounded}, one has $g^0 (x,v)\leq \widetilde{g}^0 (x,v)$ for all $(x,v) \in \R\times V$ and thus 
$g(t,x,v) \leq \widetilde{g} (t,x,v)$ for all $(t,x,v) \in \R_+ \times \R\times V$, from which the conclusion follows.

2. Let $c<c^*$. The same arguments as in the first step yield that we can assume that 
$$g^0 (x,v) = \left\{ \begin{array}{lll} \gamma M(v)  &\hbox{ if }& x<x_L\,,\\
                      0 &\hbox{ if }& x\geq x_L\,.\\ 
                      \end{array}\right. $$
Let $\sub$ a compactly supported subsolution of \eqref{eq:kinKPPc} associated with a speed $c^0 \in (c,c^*)$. Since $\sub$ can be chosen arbitrarily small, up to translation of $\sub$, we can always assume that $\sub(x,v) \leq g^0 (x,v)$. Let $g_2$ the solution of the Cauchy problem \eqref{Cauchypbm} 
associated with the initial datum $g_2 (0,x,v) = \sub(x,v)$. The comparison principle yields $g(t,x,v) \geq g_2 (t,x,v)$ for all $(t,x,v) \in \R_+ \times\R\times V$.

Let $w(t,x,v) = g_2 (t,x+c^0t,v)$. This function satisfies
\begin{equation} \label{Cauchycw} \left\{\begin{array}{l} 
 \partial_t w+  (v-c^0) \partial_x w =  M(v)\rho_{w} - w +r  \rho_{w} \left( M(v)-w \right) \quad \hbox{ in } \R_+ \times \R\times V \\
w(0,x,v) = g(x,v) \quad \hbox{ in } \R\times V.\\ \end{array} \right. 
\end{equation}
Clearly $\sub$ is a (stationary) subsolution of this equation. The comparison principle Proposition \ref{prop-mp} yields that $t\mapsto w(t,x,v)$ is nondecreasing for all $(x,v) \in \R\times V$ (see also Lemma \ref{lem-timedec} for a similar proof). 

Let $w_\infty (x,v)= \lim_{t\to +\infty} w(t,x,v)$. This function is clearly a weak solution of 
$$(v-c^0) \partial_x w_\infty =  M(v)\rho_{w_\infty} - w +r  \rho_{w_\infty} \left( M(v)-w_\infty \right) \quad \hbox{ in }  \R\times V.$$
Moreover, we have $w_\infty (x,v) \geq w(0,x,v) = \sub(x,v)$ and $w_\infty (x,v) \leq M(v)$. 

\begin{lem}[Sliding lemma]\label{lem:sliding}
We have $w_\infty \equiv M$.
\end{lem}

\begin{proof}[\bf Proof of Lemma \ref{lem:sliding}]
\noindent{\bf \# Step 1.} First we prove that $w_\infty$ is positive everywhere. 

Take $(x_0,v_0)\in\R\times V$ 
such that $w_\infty(x_0,v_0)>0$. As $\widetilde w(t,x,v)= w_\infty(x-c^0 t,v)$ satisfies 
\eqref{eq:kinKPP}, Proposition \ref{prop-mp} yields $\widetilde w(t,x,v) >0$ for all $(t,x,v) \in \R_+ \times \R\times V$ such that $|x-x_0|< v_{\rm max} t$. 
As $c^0<c^* \leq v_{\rm max}$, for all $(x,v)\in \R\times V$ one can take $t>0$ large enough so that $|x+ct-x_0| < v_{\rm max} t$. Therefore $w(x,v) = \widetilde w(t,x+ct,v)>0$. 
We thus conclude that $w_\infty$ is positive over $\R\times V$.

\medskip

\noindent{\bf \# Step 2.}  Next we prove that $\inf w_\infty >0$.

Let $y\in \R$. Define $\sub_y (x,v)= \sub(x-y,v)$, and 
$$\kappa_y= \sup \{ \kappa \in (0,1), \ w_\infty \geq \kappa \sub_y \hbox{ in } \R \times V\}\, .$$
As $\sub_y$ is compactly supported and $w_\infty$ is positive over $\R\times V$ and continuous, we have
$w_\infty \geq \kappa \sub_y  $ when $\kappa>0$ is small enough. Therefore $\kappa_y>0$. 

We argue by contradiction. Assume that $\kappa_y < 1$. The definition of $\kappa_y$ yields that $u=w_\infty - \kappa_y \sub_y\geq 0$ and that $\inf_{\R\times V} u=0$. 
As $\sub_y$ is compactly supported, this infimum is indeed reached: there exists $(x_y,v_y)\in \R\times V$ such that $u(x_y,v_y)=0$. 
Assume that $u\not\equiv 0$ and take $(x_y',v_y') \in\R\times V$ such that $w_\infty(x_y',v_y') > \kappa_y \sub_y (x_y',v_y')$. 

We introduce $w_1 (t,x,v) = w_\infty(x-c^0t,v)$ and $w_2 (t,x,v) = \kappa_y \sub_y (x-c^0t,v)$. As $w_1 (0,x_y',v_y')> w_2 (0,x_y',v_y')$, Proposition \ref{prop-mp} gives
$w_1 (t,x,v) > w_2 (t,x,v) $ for all $(t,x,v) \in \R_+ \times \R\times V$ such that $|x-x'_y|< v_{\rm max} t$, that is:
$$w_\infty(x-c^0t,v) > \kappa_y \sub_y (x-c^0t,v) \quad \hbox{ if }\, |x-x'_y|< v_{\rm max} t\,.$$
As $c^0<c^*\leq v_{\rm max}$, for all $x\in\R$, one can take $t>0$ large enough so that $|x+c^0t-x'_y|< v_{\rm max} t$, leading to $w_\infty(x,v) > \kappa_y \sub_y (x,v)$ for all $(x,v) \in \R\times V$, a contradiction 
since equality holds at $(x_y,v_y)$. 

Hence, $w_\infty \equiv \kappa_y \sub_y$, which is also a contradiction since $w_\infty$ is positive while $\sub_y$ is compactly supported. We conclude that $\kappa_y = 1$, namely  
$w_\infty \geq \sub_y$. Evaluating this inequality at $x=y$, one gets $w_\infty(y,v) \geq  \sub(0,v)$ for all $(y,v)\in \R\times V$. 
As $\inf_V g(0,v)>0$ under the assumption  $\inf_V M>0$, 
we have proved in fact that 
$$\inf_{\R\times V} w_\infty >0 \,.$$

\noindent{\bf \# Step 3.}  
As $\inf_{V} M>0$, we can define 
$$\kappa^* = \sup \{ \kappa  \in (0,1), \ w_\infty(x,v)\geq \kappa M(v) \hbox{ for all } (x,v) \in\R\times V\}.$$
We know from the previous step that this quantity is positive. If $\kappa^*<1$, then the same types of arguments as in Step 2 lead to a contradiction. Hence $\kappa^* =1$, meaning 
that $w_\infty\geq M(v)$. As $w_\infty\leq M(v)$, we conclude that $w_\infty\equiv M(v)$. 
\end{proof}

As a consequence of Lemma \ref{lem:sliding} we obtain 
$$\lim_{t\to +\infty} g_2 (t,x+c^0t,v) = M(v) \quad \hbox{for all } (x,v) \in \R\times V.$$
This implies in particular that $\lim_{t\to +\infty}g (t,x+c^0t,v) = M(v)$ for all $(x,v)$ by  a sandwiching argument. 
Moreover, as $g^0$ is nonincreasing with respect to $x$, $x\mapsto g(t,x,v)$ is nonincreasing and thus 
$g(t,x,v) \geq g(t,c^0 t,v)$ for all $x\leq c^0 t$, from which the conclusion follows since $c^0>c$. 
\end{proof}

\section{Proof of the dependence results}\label{dependance}

\begin{proof}[\bf Proof of Proposition \ref{prop:bound on c*}(a)]
Recall that the dispersion relation giving the speed $c(\lambda)$ as a function of the exponential decay $\lambda$ is $I(\lambda;c(\lambda)) = 1$, 
where $I$ is defined in \eqref{eq-defI}. Let introduce $I_\sigma$ the function associated with the dilated velocity profile 
$M_\sigma$. The function $I_\sigma$ clearly satisfies the scaling relation $I_\sigma (\lambda;c(\lambda)) = I (\sigma \lambda; \sigma^{-1} c(\lambda))$, 
therefore we get $c^*(M_\sigma)=\sigma c^* (M)$ from the very definition of $c^*$. 
\end{proof}

\begin{proof}[\bf Proof of Proposition \ref{prop:bound on c*}(b)]
We use the symmetry of the kernel $M(v) = M(-v)$ to write
\begin{equation*}
I(\lambda;c) =  \int_0^{v_{\max}} \frac{(1 + r)(1 + \lambda  c)}{(1 + \lambda c)^2 - \lambda^2v^2} 2M(v)\, dv \, .
\end{equation*}
Let define
\begin{equation*} 
f(v) =  \dfrac{(1+r)(1 + \lambda c)}{(1 + \lambda c )^2 - \lambda^2 v^2}.\,
\end{equation*}
It is an increasing function over $(0,v_{\max})$, thus $f_\star= f$. The Hardy-Littlewood inequality \cite[Chap. 3]{Lieb-Loss} yields
\[\int_0^{v_{\max}} M^\star(v)f_\star(v) dv \leq \int_0^{v_{\max}} M(v) f(v) dv \leq \int_0^{v_{\max}} M_\star(v)f_\star(v) dv\,. \]
The dispersion relation is nonincreasing with respect to $c$. It follows immediately that 
\begin{equation*}
c^* (M^\star) \leq c^*(M) \leq c^* (M_\star).
\end{equation*}
\end{proof}

\begin{proof}[\bf Proof of Proposition \ref{prop:bound on c*}(c)]
We use the symmetry of the kernel $M(v) = M(-v)$. For $\lambda>0$ the dispersion relation writes
\begin{equation}
(1 + r) \int_{0}^{v_{\max}} \frac{(1 + \lambda  c(\lambda))}{(1 + \lambda  c(\lambda))^2 - \lambda^2v^2}2M(v)\, dv = 1\, . \label{eq:dispersion-symmetric}
\end{equation}
Since the function $X\mapsto \left( (1 + \lambda  c(\lambda))^2 - \lambda^2 X \right)^{-1}$ is convex on its domain of definition, Jensen's inequality yields
\begin{equation*}
(1 + r) \frac{(1 + \lambda  c(\lambda))}{(1 + \lambda c(\lambda))^2 - \lambda^2\left( 2\int_0^{v_{\max}}v^2M(v)\, dv\right)} \leq 1\, .
\end{equation*}
We recognize the dispersion relation associated with the two-speed model \cite{Bouin-Calvez-Nadin}. We deduce
\begin{equation*}
\lambda^2  c(\lambda)^2 + (1-r) \lambda  c(\lambda) - D \lambda^2 -r  \geq 0 \, .
\end{equation*}
This second-order polynomial has a negative value at $c = 0$, therefore $ c(\lambda)$ is necessarily greater than the vanishing value,
\[  c(\lambda) \geq \dfrac{(r-1) + \sqrt{(r-1)^2 + 4( D\lambda^2 + r)}}{2\lambda}\, . \]
Minimizing with respect to $\lambda>0$, we deduce that,
\[
\begin{cases} c^* \geq \dfrac{2\sqrt{rD}}{1+r} & \mbox{if $r< 1$} \, ,\medskip\\
c^* \geq \sqrt{D}  & \mbox{if $r\geq  1$}  \, . 
\end{cases}
\]
On the other hand we clearly obtain from \eqref{eq:dispersion-symmetric},
\begin{equation*}
(1 + r) \frac{(1 + \lambda  c(\lambda))}{(1 + \lambda  c(\lambda))^2 - \lambda^2v_{\max}^2} \geq 1\, .
\end{equation*}
By comparison of the relations, as in the proof of Proposition \ref{prop:bound on c*}(b), we deduce that the speed corresponding to a given distribution $M(v)$ supported on $(-v_{\max}, v_{\max})$ is smaller than the speed corresponding to $\frac12 (\delta_{-v_{\max}}+\delta_{v_{\max}})$. This peculiar case is analysed in \cite{Bouin-Calvez-Nadin}. The minimal speed in this case is 
\[ c^*\left(\frac12 (\delta_{-v_{\max}}+\delta_{v_{\max}})\right) = \begin{cases}   \dfrac{2\sqrt{r}}{1+r}v_{\max} & \mbox{if $r< 1$} \, ,\medskip\\
 v_{\max}  & \mbox{if $r\geq  1$}  \, . 
\end{cases} \]
\end{proof}

\begin{proof}[\bf Proof of Proposition \ref{prop:bound on c*}(d).]
The dispersion relation for the rescaled equation \eqref{eq:kinKPP2} reads
\begin{equation} \label{eq-defIe} I_\e(\lambda;c) = (1 + \epsilon^2 r) \int_V \dfrac{1}{1 + \epsilon^2\lambda(c - v/\eps)} M(v)\, dv\,. \end{equation}
The previous result guarantees that $c^*_\eps$ is bounded from above for $\eps^2 r<1$,
\[ c_\eps^* \leq \dfrac{2\sqrt{\eps^2r}}{1 + \eps^2r} \left(\dfrac{v_{\max}}\eps\right) \leq  2\sqrt{ r}   v_{\max} \, . \]
For a given $\lambda>0$, we perform a Taylor expansion of \eqref{eq-defIe} up to second order,
\[I_\e(\lambda;c)=1+\epsilon^2 (r-\lambda c +\lambda^2D) +O(\e^3)\,,\]
uniformly for $c\in [0,2\sqrt{ r}   v_{\max}]$, since $V$ is bounded. Therefore, solving the relation dispersion for the minimal speed boils down to solving 
\[ r-\lambda c_\eps(\lambda) +\lambda^2D + O(\eps) = 0\, . \]
We deduce
\[ \lim_{\eps\to 0} c_\eps(\lambda) = \frac r\lambda + \lambda D\, . \]
Therefore the minimal speed verifies $\lim_{\eps\to 0} c_\eps^* = 2\sqrt{rD}$.  
%
%
\end{proof}


\section{Stability of the travelling waves}\label{Stability}

\subsection{Linear stability}

In this Subsection, we focus on the linearized problem around some travelling wave moving at speed $c\in [c^* , v_{\rm max})$. 
We recall that  we consider a solution $u$ of the equation associated with the linearization around a travelling wave:
\begin{equation}\label{kinmf}
\partial_t u + (v - c) \partial_{ z } u + \left( 1 + r \rho_f \right) u = \left( (1+r) M  - r f   \right) \int_V u'dv'.
\end{equation}
where the notation $'$ always stands in the sequel for a function of the $(t, z ,v')$ variable. 
We shall prove stability of the wave in a suitable $L^2$ framework, inspired by \cite{Kirchgassner92,Gallay94,GallayRaugel,Bouin-Calvez-Nadin}.

\begin{proof}[\bf Proof of Theorem \ref{LinStability}]
\eb{We search for an ansatz $u=w e^{\phi}$}. The function $w$   satisfies: 
\begin{equation}\label{eq-u}
 \partial_t w +  (v - c) \partial_{ z } w + \left(   (v -  c) \partial_ z  \phi +  1 + r  \rho_f  \right) w = \left((1+r) M  - r f   \right) \int_{V} e^{\phi' - \phi} w' dv',
\end{equation}
From \eqref{eq-u}, we shall derive the dissipation inequality \eqref{eq:dissipation linear}. We test \eqref{eq-u} against $w$ to obtain the \eb{following} energy estimate:
\begin{multline}\label{kinNRJ}
\frac{d}{dt} \left( \frac12 \int_{\R \times V}   \left\vert w \right\vert^2 d  z  dv \right) + \int_{\R \times V} \left(   (v -  c) \partial_ z \phi +   1 + r  \rho_f \right) \left\vert w \right\vert^2 d  z  dv \\ = \int_{\R \times V \times V'} \left((1+r) M  - r f   \right) e^{\phi' - \phi} w w' dvdv'd z .
\end{multline}
We aim at choosing a weight $\phi$ such that the dissipation is coercive in $L^2$ norm. Let define the symmetric kernel $K$ as follows
\begin{equation}\label{eq:defT(v,v')}
K(v,v') = \left(  (v -  c) \partial_ z \phi + 1 + r\rho_f  \right) \delta_{v = v'} -\dfrac12 \left( \left( ( 1 + r  )M - r f\right)  e^{\phi' - \phi} +  \left( ( 1 + r  )M' - r f'\right)  e^{\phi - \phi'} \right),
\end{equation}
we seek a function $\phi$ such that 
\begin{equation*}
K(v,v') \geq A( z ,v) \delta_{v = v'},
\end{equation*}
for a suitable positive function $A$, in the sense of kernel operators. For this purpose we focus on the eigenvalues of the kernel operator $A( z ,v) \delta_{v = v'} - K(v,v') $. 

\begin{defi}[Weight $\phi$] \label{eq:def weight phi}
We introduce
$\Lambda( z ) \in \left[ 0 , \frac{1}{  v_{\max} -  c} \right) $ the smallest solution of the following dispersion relation 
\begin{equation}\label{eq:def-Lambda}
\int_{V} \frac{ (1 + r) M(v) - r  f( z ,v) }{ 1 + \Lambda( z )  (c-v) }\, dv = 1\,,
\end{equation}
and we define $\Gamma( z )$ through the differential equation
\begin{equation}\label{eq:def-C}\frac12 \frac{\Gamma'( z )}{\Gamma( z )} = \Lambda( z )\, ,\quad \Gamma(0) = 1\, .\end{equation}
Finally we define 
\begin{equation}\label{eq:def-phi}
\phi( z ,v) = \frac12 \ln \left( \dfrac{ (1 + r) M(v) -  r  f( z ,v)}{\Gamma( z )}   \right) \,,
\end{equation}
\end{defi}
Recall that $0\leq f \leq M$, so that the weight $\phi$ is well-defined as soon as $\Lambda$ is well-defined. 
A small argumentation is required to prove that $\Lambda( z )$ is well-defined too. For a given $c$ and $ z $, define
\begin{equation*}
G(\Lambda) = \int_{V} \frac{ (1+r) M(v) - r  f( z ,v) }{ 1 + \Lambda  (c -  v) }\, dv\,, \quad \Lambda \in \left[ 0 , \frac{1}{  v_{\max} -  c} \right).
\end{equation*}
The function $G$ is continuous, and satisfies the following properties, 
\begin{align*}
& G(0) = \left( 1 + r \right) - r \rho_f( z ) = (1 + r) \left( 1 - \rho_f( z ) \right) + \rho_f( z ) \in \left[ 1 , 1 + r  \right]\,,\\
& G(\lambda ) = \int_{V} \frac{ (1+r) M(v) - r   f( z ,v)  }{ 1 + \lambda  (c -  v) } dv = 1 - \int_{V} \frac{r  f( z ,v) }{ 1 + \lambda  (c-v) } dv \leq 1\,,
\end{align*}
where $\lambda$ is chosen such that $I(\lambda;c) = 1$.  
Thus we can define the smallest $\Lambda( z )\in \left[ 0 , \lambda \right]$ such that $G(\Lambda( z ))=1$.

\eb{
\begin{rmq}
[Asymptotic behavior of the weight] It is important to state clearly the asymptotic behavior of the weight as it determines the possible perturbations. The following estimates were established in Section~\ref{existence}, 
\begin{equation*}
\lim_{z \to - \infty} f(z,v) = M(v), \quad \lim_{z \to + \infty} f(z,v) = 0.
\end{equation*}
We deduce from \eqref{eq:def-Lambda} and the dispersion relation \eqref{eq:dispersion} that 
\begin{equation*}
\lim_{z \to - \infty} \Lambda(z) = 0, \quad \lim_{z \to + \infty} \Lambda(z) = \lambda.
\end{equation*}
It yields from \eqref{eq:def-phi} and \eqref{eq:def-C} that 
\begin{equation*}
\phi(z,v) \underset{z \to -\infty}{\sim} - \frac12 \log\left(\dfrac{\Gamma( z )}{ M(v) }\right), \quad  \phi(z,v)   \underset{z \to +\infty}{\sim} -  \lambda z\, .   
\end{equation*}
The precise behavior of $\Gamma$ near $-\infty$ would require further analysis about the integrability of $\Lambda$. However, we believe it converges towards a positive constant. As compared to \cite{Cuesta12}, $\phi$ combines the two weights in a single one, see \cite[Eqs (3.6)-(3.7)]{Cuesta12}. As a consequence, the perturbation $g^0 - f$ must decay faster than the wave profile at $+\infty$ to have finite energy, as usual.
\end{rmq}
}

\begin{lem}\label{eigens} 
Let $A$ be defined as 
\[ A(  z  , v ) = \frac{r}{2} \left( \rho_f( z ) + \frac{ f( z ,v)}{(1+r)M(v) - r f( z ,v)} \right)\,,\] and $\bT$ be the operator associated with the symmetric kernel $T(v,v') = A( z ,v) \delta_{v = v'} - K(v,v')$. The operator  $\bT$ is nonpositive.
\end{lem}

\begin{proof}[\bf Proof of Lemma \ref{eigens}]
We shall prove that $0$ is the Perron eigenvalue of the  operator  $\bT$. For that purpose we shall exhibit a positive eigenvector in the kernel of $\bT$.
The equation $\bT(W)=0$ reads
\begin{equation*} 
(\forall v\in V)\quad \int_{V} \left( A( z ,v) \delta_{v = v'} - K(v,v')\right) W(v') dv' = 0\, .
\end{equation*}
Plugging the formula for $K(v,v')$ \eqref{eq:defT(v,v')} into this expression we get, 
\begin{multline*}
\left(  A( z ,v) -  (v -  c) \partial_ z \phi( z ,v) - 1 - r  \rho_f( z ) \right) W(v) + \frac12 \left( (1+r)  M(v) - r   f( z ,v)  \right) \left(\int_V e^{\phi( z ,v') - \phi( z ,v)} W(v')\, dv'\right)
\\
+  \frac12 \int_{V} \left( (1+r)  M(v') - r   f( z ,v')  \right) e^{\phi( z ,v) - \phi( z ,v')} W(v')\, dv' = 0\, .
\end{multline*} 
From the Definitions \eqref{eq:def-Lambda}-\eqref{eq:def-phi} we have,
\begin{equation*}
\partial_{ z } \phi( z ,v) = - \frac{r}{2} \frac{\partial_{ z }  f ( z ,v)}{ (1+r) M(v) - r   f( z ,v) } -  \Lambda( z )\,. 
\end{equation*}
The weight $\phi$ and the function $A$ are chosen such that
\begin{align*} 
& A( z ,v) -  (v -  c) \partial_ z \phi( z ,v) - 1 - r  \rho_f( z )\\ 
& = \frac{r}{2} \left( \rho_f( z ) + \frac{ f( z ,v)}{(1+r)M(v) - r f( z ,v)} + (v-c)  \frac{\partial_{ z }  f ( z ,v)}{ (1+r) M(v) - r   f( z ,v) } \right) + (v-c)  \Lambda( z ) - 1 - r\rho_f( z )\\
& =  \frac{r}{2} \left( 2\rho_f( z )\right) + (v-c)\Lambda( z ) - 1 - r\rho_f( z )\\
& =   (v-c)\Lambda( z ) - 1  \,.
\end{align*}
Therefore the equation $\bT(W) = 0$ is equivalent to
\begin{equation}\label{eq:eigen-u}
W(v) = \dfrac1{1 + \Lambda( z )(c-v)}\left( \frac12 \left( (1+r)  M(v) - r   f( z ,v)  \right) e^{-\phi( z ,v)}  X_1( z ) + \frac12 e^{\phi( z ,v)} X_2( z ) \right)\,,
\end{equation}
where the macroscopic quantities $X_1$ and $X_2$ are defined as follows,
\[ X_1( z ) = \int_V e^{ \phi( z ,v')} W(v')\, dv'\, , \quad X_2( z ) = \int_{V} \left( (1+r)  M(v') - r   f( z ,v')  \right) e^{ - \phi( z ,v')} W(v')\, dv'\, . \]
To resolve this eigenvalue problem, we seek proper values for $X_1$ and $X_2$. From \eqref{eq:eigen-u} we deduce that they 
are solution of a $2\times2$ closed linear system, namely
\[
\begin{cases}
\displaystyle X_1( z ) = \frac12 \left(  \int_{V} \frac{  (1+r)  M(v)  - r   f( z ,v)   }{ 1 + \Lambda( z ) (c -  v)  }\, dv \right) X_1(v) + \frac12 \left(  \int_{V} \frac{ e^{2\phi( z ,v)}   }{ 1 + \Lambda( z ) (c -  v)  }\, dv \right) X_2( z )\medskip\\
\displaystyle X_2( z ) = \frac12 \left(  \int_{V} \frac{ \left( (1+r)  M(v)  - r   f( z ,v) \right)^2  e^{-2\phi( z ,v)}  }{ 1 + \Lambda( z ) (c -  v)  }\, dv \right) X_1(v) +  \frac12 \left(  \int_{V} \frac{  (1+r)  M(v)  - r   f( z ,v)   }{ 1 + \Lambda( z ) (c -  v)  }\, dv \right) X_2( z )
\end{cases}
\]
This system simplifies thanks to the choice of $\Lambda( z )$ \eqref{eq:def-Lambda}. Indeed we have
\begin{align*}
&\int_{V} \frac{  (1+r)  M(v)  - r   f( z ,v)   }{ 1 + \Lambda( z ) (c -  v)  }\, dv   = 1 \\
&\int_{V} \frac{ e^{2\phi( z ,v)}   }{ 1 + \Lambda( z ) (c -  v)  }\, dv = \left( \int_{V} \frac{  (1+r)  M(v)  - r   f( z ,v)   }{ 1 + \Lambda( z ) (c -  v)  }\, dv \right) \dfrac1{\Gamma( z )} = \dfrac1{ \Gamma( z )} \\
&  \int_{V} \frac{ \left( (1+r)  M(v)  - r   f( z ,v) \right)^2  e^{-2\phi( z ,v)}  }{ 1 + \Lambda( z ) (c -  v)  }\, dv  = \left(    \int_{V} \frac{  (1+r)  M(v)  - r   f( z ,v)   }{ 1 + \Lambda( z ) (c -  v)  }\, dv \right)  \Gamma( z )  =  \Gamma( z )\, . 
\end{align*}
We are reduced to the following eigenvalue problem,
\[
\begin{pmatrix}
X_1( z ) \\
X_2( z )
\end{pmatrix} = 
\frac12\begin{pmatrix}
1 & \Gamma( z )^{-1} \\
\Gamma( z ) & 1
\end{pmatrix} \begin{pmatrix}
X_1( z ) \\
X_2( z )
\end{pmatrix} \, .
\]
Clearly, $(X_1( z ),X_2( z )) =  (1, \Gamma( z ))$ is the unique solution up to multiplication. We obtain eventually that $W$ is given (up to a multiplicative factor) by
\begin{align*}
W(v) & = \frac12 \frac{  \left( (1+r)M(v) - r   f( z ,v)   \right) e^{- \phi( z ,v)}   + e^{\phi( z ,v)} \Gamma( z ) }{ 1 + \Lambda( z ) (c -  v) } \\
& = \dfrac{\left[ \left( (1+r)M(v) - r   f( z ,v)   \right) \Gamma( z ) \right]^{1/2}}{1 + \Lambda( z ) (c -  v) } > 0 \,.
\end{align*}
As a consequence, we have found that the symmetric operator $\bT$, which is nonnegative out of the diagonal $v = v'$, possesses a positive eigenvector $W$ associated with the eigenvalue 0. Therefore it is a nonpositive operator. This ends the proof of the Lemma.
\end{proof}

We can now conclude the proof of Theorem \ref{LinStability}. Lemma \ref{eigens} claims that for all $w \in L^2\left( \R \right)$ such that $u = e^\phi w$ is solution to the linearized equation, we have
\begin{equation*}
\frac{d}{dt} \left( \frac12 \int_{\R \times V}  \left\vert w \right\vert^2  d z  dv \right) + \int_{\R \times V} A( z ,v) \left\vert w \right\vert^2 d  z  dv \leq 0\,.
\end{equation*}
which proves the Proposition.
\end{proof}

\begin{rmq}[Non optimality of the weight]\label{rk:phi not optimal}
We believe that the weight $\exp(\phi( z ,v))$ proposed in Definition \ref{eq:def weight phi} is not optimal with respect to the spectral property of the linearized operator \eqref{kinmf}. Indeed the dissipation factor $A( z ,v)$ is equivalent in the diffusion limit ($r\to r \eps^2$) to $r\eps^2\rho_f( z )$, although we expect $2 r\eps^2 \rho_f( z )$ \cite{Kirchgassner92, Cuesta12}. The missing factor $2$ is responsible for the restriction $\gamma>1/2$ in our nonlinear stability result, Corollary \ref{NonLinStability}. 

Let us recall how to derive the spectral properties of the linearized equation in the diffusive limit, namely the linearized Fisher-KPP equation, 
\begin{equation}\label{eq:KPP linearized} \partial_t u - c\partial_ z  u - D \partial_{ z  z } u = r (1 - 2\rho_f) u\, ,  \end{equation}
where $\rho_f( z )$ is the profile of the travelling wave in the frame $ z  = x - ct$. Applying the same procedure as in the proof of Theorem \ref{LinStability}, we shall derive an equation for the weighted perturbation $w = e^{-\phi}u$, and optimize the dissipation with respect to the weight $\phi$ (see also \cite{Bouin-Calvez-Nadin}), as follows
\begin{align*}
&\partial_t w - c \partial_ z  w - D\partial_{ z  z } w - 2D \partial_ z \phi \partial_ z  w - D\partial_{ z  z } \phi w - (c \partial_ z  \phi + D|\partial_ z  \phi|^2 ) w = r(1 - 2\rho_f) w \\
& \dfrac d{dt}\left( \frac12 \int_\R |w|^2\, d z  \right) + D \int_\R |\partial_ z  w|^2\, d z  + \int_\R\left( 2r\rho_f  - r -c \partial_ z  \phi - D |\partial_ z  \phi|^2 \right) |w|^2 \, d z  = 0
\end{align*}
The best choice is achieved when $\partial_ z  \phi$ is constant and minimizes $r + c\lambda + D\lambda^2$, {\em i.e. $\partial_ z  \phi = - c/(2D)$}. In the case of the minimal speed $c = c^* = 2\sqrt{rD}$, we obtain the following dissipation formula for the linearized operator,
\begin{equation}\label{eq:dissip diff} \dfrac d{dt}\left( \frac12 \int_\R |w|^2\, d z  \right) + D \int_\R |\partial_ z  w|^2\, d z  + \int_\R 2r\rho_f  |w|^2 \, d z  = 0\, . \end{equation}
Notice the factor $2$ which is apparently missing in the dissipation term \eqref{eq:dissipation linear}. 

%
%

A systematic way to find the correct weight is to derive the eigenvectors  of the operator and its dual, then to use the framework of relative entropy (see \cite{Michel-Mischler-Perthame-JMPA} for a general presentation). This was done by Kirchg\"assner \cite{Kirchgassner92} who derived the so-called {\em eichform} for  \eqref{eq:KPP linearized}. The linearized operator ${\cal L}(u) = - c\partial_ z  u - \partial_{ z  z } u - (1 - 2\rho_f) u$ possesses obviously the nonpositive eigenvector $\eta = \partial_ z  \rho_f$, ${\cal L}(\eta) = 0$. The dual operator ${\cal L^*}(\varphi) = + c\partial_ z  \varphi - \partial_{ z  z } \varphi - (1 - 2\rho_f) \varphi$ possesses the nonpositive eigenvector $\psi =  \partial_ z  \rho_f e^{c z }$, ${\cal L^*}(\psi) = 0$, as can be checked by direct calculation. Therefore the relative entropy identity for the convex function $H(p) = \frac12|p|^2$ writes for the linearized system as follows,
\[ \dfrac d{dt}\left( \frac12 \int_\R \psi( z )  \left(\dfrac{u(t, z )}{\eta( z )} \right)^2 \eta( z )\, d z    \right) + \int_\R \psi( z ) \left|\dfrac{\partial}{\partial  z }\left(\dfrac{u(t, z )}{\eta( z )}\right) \right|^2 \eta( z )\, d z  = 0\, ,\]
which is equivalent to \eqref{eq:dissip diff} after straightforward computation (recall $w = e^{(c/2D) z }u$).

A similar strategy could be performed here: the linearized operator ${\cal L}(u) =  (v - c) \partial_{ z } u + \left( 1 + r \rho_f \right) u - \left( (1+r) M  - r f   \right) \int_{V} u' dv'$ possesses the nonpositive eigenvector $\eta = \partial_ z   f$, ${\cal L}(\eta) = 0$ (recall $ z \mapsto  f( z ,v)$ is nonincreasing). To derive the corresponding relative entropy identity, we should find an eigenvector $\psi$ in the nullset of the dual operator
\[ {\cal L^*}(\varphi) =  - (v - c) \partial_{ z } \varphi + \left( 1 + r \rho_f \right) \varphi -  \int_{V} \left( (1+r) M'  - r f'   \right) \varphi' dv'\, .  \]
Existence of such an eigenvector would follow from the Krein-Rutman Theorem. However we were not able to find an explicit formulation of $\psi$, and thus of the dissipation, which is necessary to derive a quantitative nonlinear stability estimate such as Corollary \ref{NonLinStability}. This is the reason why we stick to the weight proposed in Definition \ref{eq:def weight phi} although we believe it is not the optimal one. 
\end{rmq}


\subsection{Nonlinear stability by a comparison argument.}

%
%
\begin{proof}[\bf Proof of Corollary \ref{NonLinStability}]
First, the comparison principle of Proposition \ref{prop-mp} and \eqref{initialcond} yield 
\begin{equation}\label{comparison}
\rho_u (t,z) \geq \left( \gamma - 1 \right) \rho_f(t,  z  ), \quad \forall (t, z ) \in \R_+ \times \R.
\end{equation}
Now, we write the nonlinear equation verified by the weighted perturbation $w = e^{-\phi}u$,
\begin{equation}\label{eq-uNL}
 \partial_t w +   (v -  c) \partial_{ z } w + \left(  (v -  c) \partial_ z \phi + 1 + r  \rho_f \right) w = 
\left( (1+r)M  -rf \right) \int_{V} e^{\phi' - \phi} w' dv' - r  w \rho_u\, ,
\end{equation}
and as for the linear stability problem we test \eqref{eq-uNL} against $w$:
\begin{multline*}
\frac{d}{dt} \left( \int_{\R \times V}  \frac{\left\vert w \right\vert^2}{2} d  z  dv \right) + \int_{\R \times V} 
\left(  (v -  c) \partial_ z \phi + 1 + r  \rho_f \right) \left\vert w \right\vert^2 d  z  dv \\ 
= \int_{\R \times V \times V'} w \left( ( 1 + r  ) M - r   f \right) e^{\phi' - \phi} w' dz dvdv'  - \int_{\R \times V} r  |w|^2 \rho_u d z  dv\,.
\end{multline*}
Using \eqref{comparison} we deduce 
%
%
%
\begin{multline*}
\frac{d}{dt} \left( \int_{\R \times V}  \frac{\left\vert w \right\vert^2}{2} d  z  dv \right) + \int_{\R \times V} \left(  (v -  c) \partial_ z \phi + 1 + \gamma r  \rho_f \right) \left\vert w \right\vert^2 d  z  dv \\ \leq \int_{\R \times V \times V'} w \left( ( 1 + r  ) M  - r   f \right) e^{\phi' - \phi} w' dvdv'd z .
\end{multline*}
This last equation is very similar to \eqref{kinNRJ}. Following the same steps as in the proof of Theorem \ref{LinStability}, we find that using the same weight $\phi$ and setting 
\[ A( z  , v) = \frac{r}{2} \left( (2 \gamma - 1) \rho_f + \frac{ f}{(1 + r ) M(v) - r   f} \right)\, ,\]
we obtain an estimate very similar to \eqref{eq:dissipation linear},
\begin{equation}\label{nonlinest}
\frac{d}{dt} \left( \frac12\int_{\R \times V} \left\vert u \right\vert^2 e^{- 2 \phi( z ,v)} d  z  dv \right) + \int_{\R \times V} \frac r2 \left[  ( 2 \gamma - 1 ) \rho_f + \frac{ f}{M(v) + r \left( M(v) -  f \right)} \right] \left\vert u \right\vert^2 e^{- 2 \phi( z ,v)} d  z  dv \leq 0\,.
\end{equation}
%
%
\end{proof}


\section{Numerics} \label{sec:numerics}

In this Section, we show the outcome of numerical simulations to illustrate our results, and to motivate the last Section about accelerating fronts. We use  a simple explicit numerical scheme for approximating \eqref{eq:kinKPP}. The free transport operator is discretized using an upwind scheme.

We show in Figure \ref{fig:vborne} the expected asymptotic behavior when the velocity space is bounded. The solution of the Cauchy problem converges towards  a travelling front with minimal speed. 

Next we  investigate  the case  $V = \R$. 
Of course, numerical simulations require that the support of $M$ is truncated. We opt  for the following strategy: the velocity set is truncated $V_A = [-A,A]$, and the distribution $M$ is renormalized accordingly. For any $A>0$ we observe the asymptotic regime of a travelling front with finite speed, as expected. However, the  asymptotic spreading speed diverges as $A\to +\infty$. In fact, we observe that the envelope of the spreading speed scales approximately as $\langle c\rangle = \mathcal O ( t^{1/2} )$. Hence the front is accelerating like the power law $\langle x\rangle = \mathcal O (t^{3/2} )$. 

 \begin{figure}
 \begin{center}
 \includegraphics[width = \linewidth,height = 0.4\linewidth]{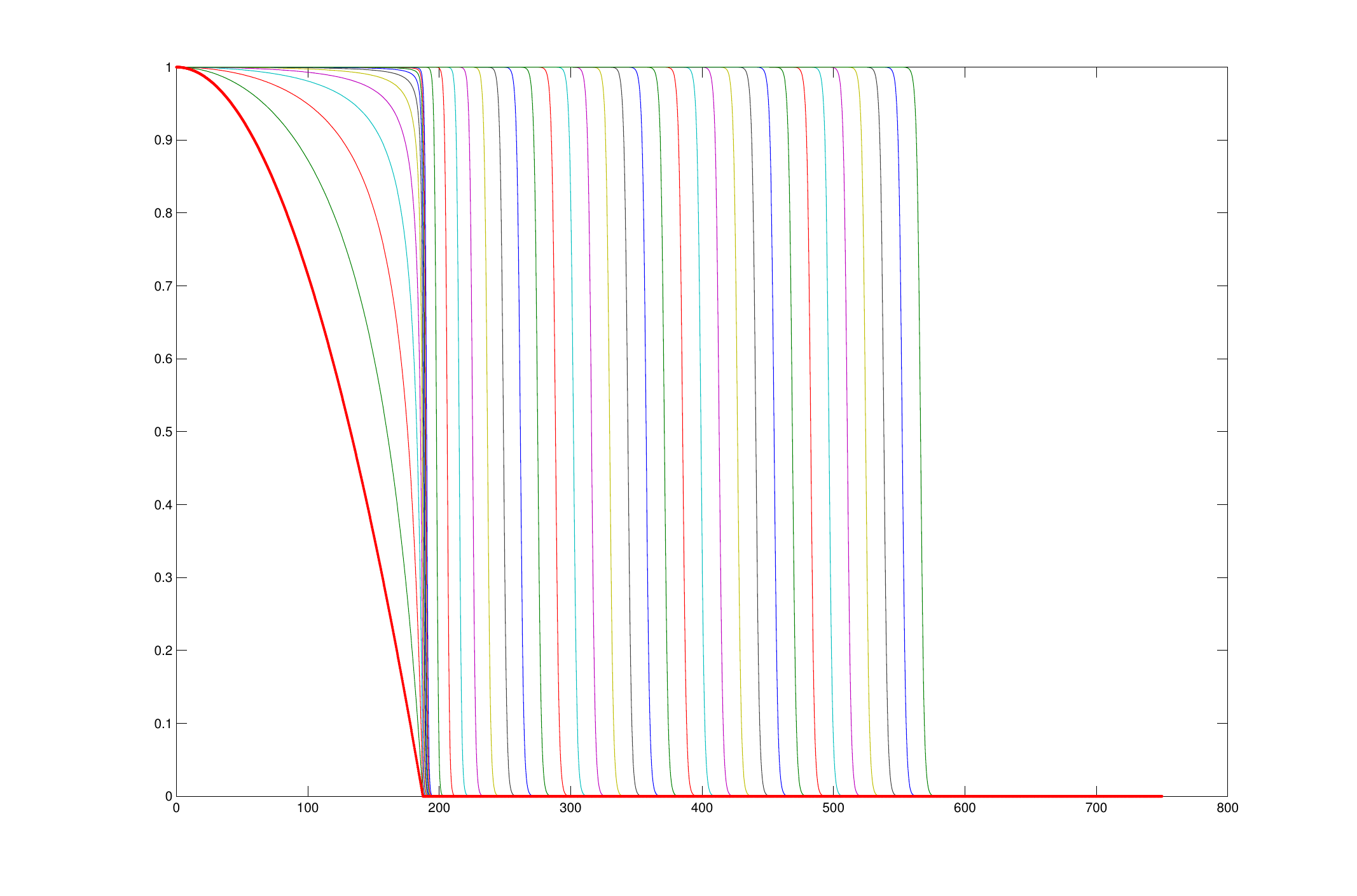}
 \caption{Numerical simulation of equation \eqref{eq:kinKPP} with initial datum being chosen as  $g^0(x<0,v) = M(v)$ and $g^0(x>0,v) = M(v) \left( 1 - \alpha x ^2 \right)_+$. The density distribution $M(v)$ is a truncated Gaussian function on a compact velocity set. We plot the evolution of the macroscopic density $\rho_g$ (initial condition in red bold). After short time the density has accumulated towards a steep profile. Then the front starts to propagate with constant speed. \label{fig:vborne}}
 \end{center} 
 \end{figure}

%
%
 \begin{figure}
 \begin{center}
 \includegraphics[width = .49\linewidth]{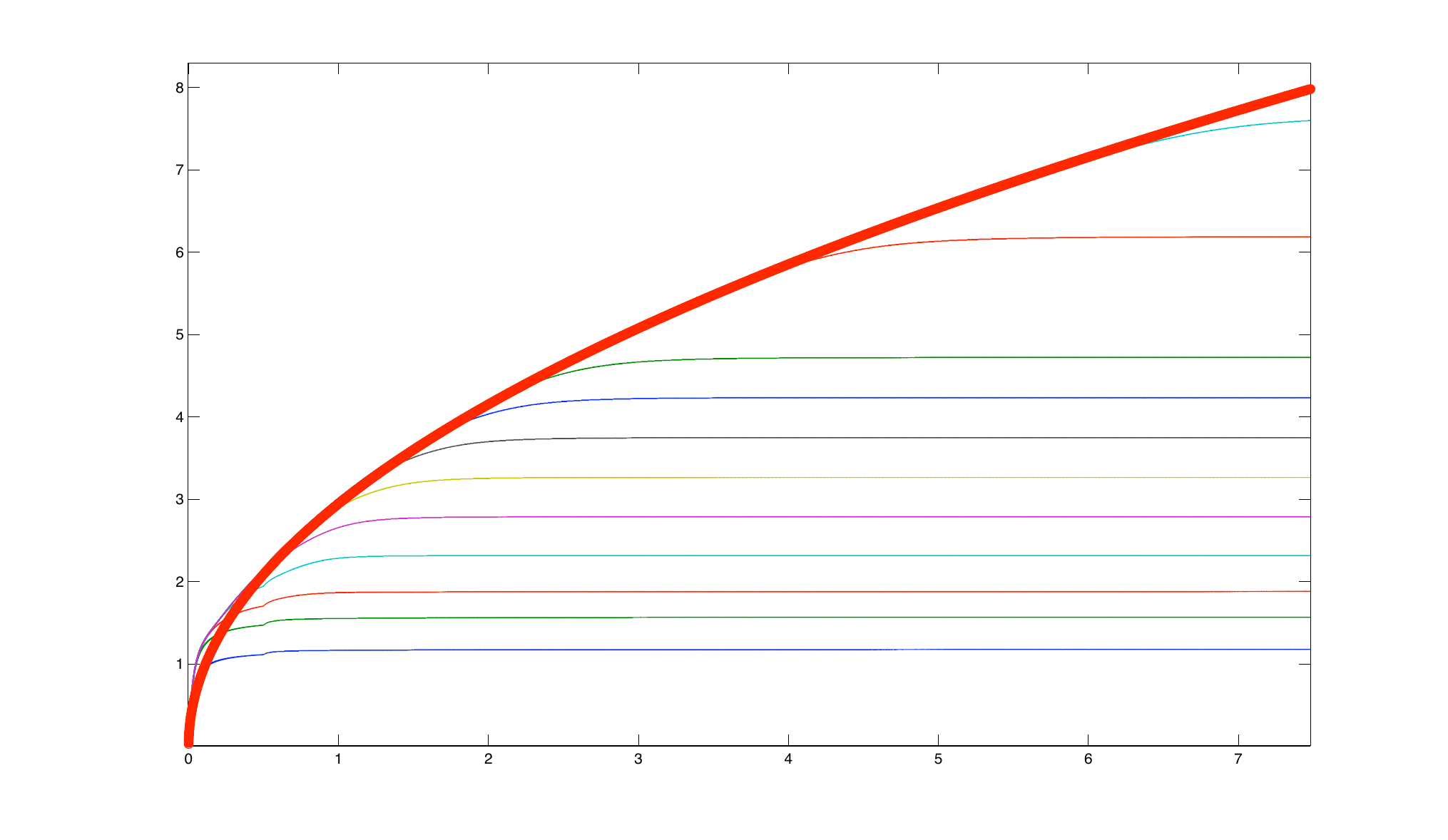} \; 
 \includegraphics[width = .49\linewidth]{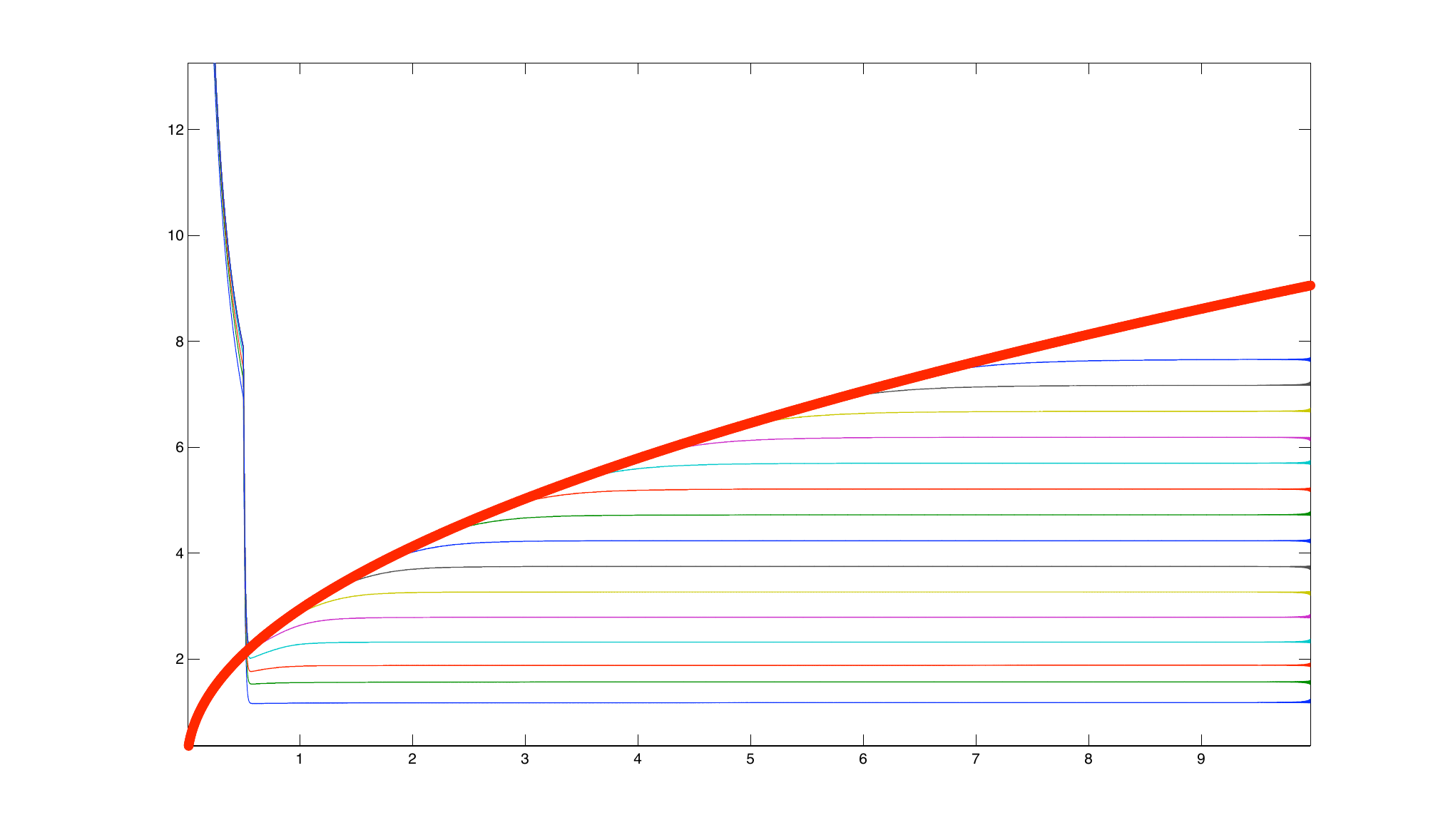}
 \caption{Numerical simulations of equation \eqref{eq:kinKPP} with initial datum being chosen as (left) $g^0(x<0,\cdot) = M(\cdot)$ and $g^0(x>0,\cdot) = 0$, and (right) the same initial condition as in Figure \ref{fig:vborne}. The distribution $M$ is a Gaussian function. Each plot corresponds to the evolution of speed of the front for some truncation $V = [-A,A]$, for (left) $A = [(1:9),15,20]$, and (right) $A = (1:15)$. The curves are ordered from bottom to top: the speed of the front increases with $A$. We plot in red bold the function $t\mapsto t^{1/2}$. We observe that it fits very well with the envelop of the family of curves. As a consequence, the front propagation scales approximately as $\langle x\rangle = \mathcal O ( t^{3/2} )$.   \label{fig:vnonborne}
 }
 \end{center}
 \end{figure}
 
 \begin{figure}
 \begin{center}
 \includegraphics[width = \linewidth,height = 0.4\linewidth]{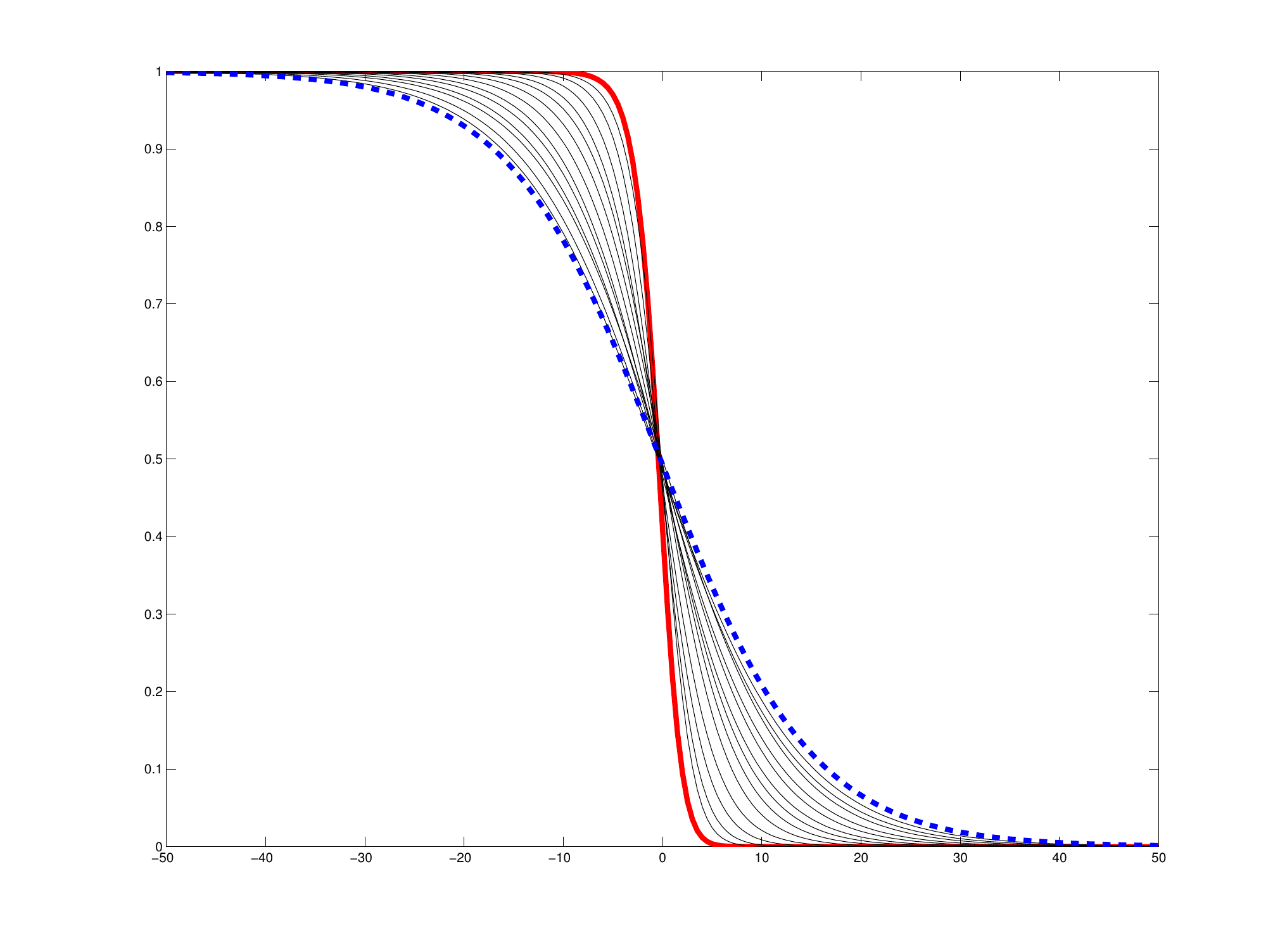}
 \caption{Same numerical simulation as Figure \ref{fig:vnonborne} with the same initial datum as in Figure \ref{fig:vborne}. We superpose various macroscopic profiles $\rho_g$ obtained after long time simulations of the scheme, for different truncation levels $A = (1:15)$. Time   is the same for all profiles. It is sufficiently large to guarantee that we have reached the asymptotic regime (Figure \ref{fig:vnonborne}, right). All profiles are translated such that $\rho_g(T,0) = \frac12$.   We observe that the exponential decay is monotonically decreasing with $A$. This indicates that the solution corresponding to $V = \R$ should  flatten when $t \to \infty$. }
 \end{center}
 \end{figure}


\section{Superlinear spreading and accelerating fronts ($V=\R$)}

We assume in this Section that $V=\R$ and that $M(v) > 0$ for all $v \in \R$. We prove superlinear spreading. We deduce as a Corollary that there cannot exist a travelling wave solution of \eqref{eq:kinKPP}. We also give some quantitative features about the spreading of the density \eb{when $M$ is a Gaussian distribution. In accordance with numerical simulations, we prove a sharp spreading rate, namely $\mathcal{O}\left(t^{3/2}\right)$. To this end, we construct explicit sub- and supersolutions from which we  estimate the spreading (respectively from below and above). 
} 

Before we go to the proof, let us give some heuristics concerning the superlinear spreading rate. Reaction-diffusion fronts with KPP nonlinearity are {\em pulled fronts}: the spreading rate is determined by the dynamics of small populations at the far edge of the front. In the kinetic model with unbounded velocities, individuals with arbitrary large speeds go at the far edge of the front. There, their density grows exponentially, and pull the accelerating front. 


\subsection{\eb{Nonexistence} of travelling waves and superlinear spreading}

\begin{proof}[\bf Proof of Proposition \ref{prop:spreadingunbounded}.]

Let $\underline{A}>0$ so that $(1+r)\int_{-\underline{A}}^{\underline{A}} M(v)dv=1$. For all $A> \underline{A}$ we define the renormalized truncated kernel and the associated growth rate,
$$M_A (v) = \frac{1_{[-A,A]}(v)}{\int_{-A}^A M} M(v)  \quad \hbox{ and } \quad r_A =(1+r)\int_{-A}^A M(v)dv  -1 \in (0, r)\, .$$
As $M_A$ is compactly supported, we can apply the results proved when $V$ is bounded in order to construct appropriate subsolutions. 

Before we proceed with subsolutions we investigate the dispersion relation in the limit $A\to +\infty$. 
Define for all $c \in (0,A)$ and $\lambda \in \left(0, 1/ (A-c)\right)$:
$$I_A (\lambda;c)= (1+r_A)\int_{\R} \frac{M_A(v)}{1+\lambda (c-v)}dv = (1+r) \int_{-A}^A \frac{M(v)}{1+\lambda (c-v)} dv$$
and $c^*_A$ the corresponding minimal speed defined in Lemma \ref{minlambda}.

\begin{lem}\label{infspeed}
One has $\lim_{A\to +\infty} c^*_A =+\infty$. 
\end{lem}

\begin{proof}[\bf Proof of Lemma \ref{infspeed}]
For all $A>\underline{A}$, let $\lambda_A \in \left( 0, 1/ (A-c^*_A)\right)$ such that 
\begin{equation} \label{eq:divcA} I_A (\lambda_A;c^*_A)= (1+r) \int_{-A}^A \frac{M (v)  }{ 1+ \lambda_A (c^*_A-v)} dv=1. \end{equation} 
If $c^*_A$ does not diverge to $+\infty$ as $A\to +\infty$, then 
it is bounded along a sequence $(A_n)_n$ and one has $\lim\lambda_{A_n} = 0$ simply by comparison $\lambda_{A_n}\leq 1/ (A_n-c^*_{A_n})$. Applying Fatou's lemma to \eqref{eq:divcA}, one gets
$$(1+r) \int_\R \liminf_{n\to+\infty} \frac{M (v) 1_{(-A_n,A_n)} (v) }{ 1+ \lambda_{A_n} (c^*_{A_n}-v)} dv = 
(1+r) \int_\R M(v)dv = 1+r \leq \liminf_{ n\to +\infty} I_{A_n}(\lambda_{A_n},c^*_{A_n}) = 1\,,$$
a contradiction. 
\end{proof}

Let $g_A$ the solution of 
\begin{equation} \label{eq:uA} \left\{ \begin {array}{ll}
\partial_t g_A + v\partial_x g_A = M_A (v) \rho_{g_A}-g_A + r_A  \rho_{g_A} \left( M_A (v) - g_A\right) \quad &\hbox{in } \R_+ \times \R\times [-A,A],\smallskip\\
g_A (0,x,v) = g^0 (x,v) \quad &\hbox{in } \R\times [-A,A],
                                       \end{array}\right. 
\end{equation}
and $\widetilde{g}_A = \frac{r_A}{r} g_A$. 
Clearly, $M_A (v) \leq \frac{M(v)}{\int_{-A}^A M}$ for all $v\in V$. Hence, multiplying \eqref{eq:uA} by $\frac{r_A}r$, we get
\begin{align*} 
\partial_t \widetilde{g}_A + v\partial_x \widetilde{g}_A 
&\leq  \frac{M(v)}{\int_{-A}^A M} \rho_{\widetilde{g}_A}-\widetilde{g}_A 
+ r_A  \rho_{\widetilde{g}_A} \left( \frac{M(v)}{\int_{-A}^A M} - g_A\right)\\
   &\leq  (1+r_A) \frac{M(v)}{\int_{-A}^A M} \rho_{\widetilde{g}_A}-\widetilde{g}_A - r_A  \rho_{\widetilde{g}_A} g_A\\
&= (1+r)M(v) \rho_{\widetilde{g}_A}-\widetilde{g}_A - r_A  \rho_{\widetilde{g}_A} g_A\\
&=(1+r)M(v) \rho_{\widetilde{g}_A}-\widetilde{g}_A - r  \rho_{\widetilde{g}_A} \widetilde{g}_A \\
&= M(v) \rho_{\widetilde{g}_A} - \widetilde{g}_A + r \rho_{\widetilde{g}_A} \left(M(v) - \widetilde{g}_A\right).
  \end{align*} 
Extending $\widetilde{g}_A$ by $0$ outside of $\R_+ \times \R\times [-A,A]$, as $\widetilde{g}_A (0,x,v) =\frac{r_A}{r}g^0 (x,v) \leq g^0 (x,v)$, 
we get that $\widetilde{g}_A$ is a subsolution of \eqref{eq:kinKPP} 
and it follows from the maximum principle stated in Proposition \ref{prop-mp} that $g(t,x,v) \geq \widetilde{g}_A (t,x,v)$  for all 
$(t,x,v) \in \R_+ \times \R\times \R$. 

On the other hand, we know from Proposition \ref{prop:spreadingbounded} that for all $c<c^*_A$:
$$\lim_{t\to +\infty} \left( \sup_{x\leq ct} |M_A (v) -\widetilde{g}_A (t,x,v)|\right)=\lim_{t\to +\infty} \left( \sup_{x\leq ct} \left(M_A (v) -\widetilde{g}_A (t,x,v)\right)\right)=0.$$ 
Hence, as $M(v)\geq g(t,x,v) \geq \widetilde{g}_A (t,x,v)$
and $M_A (v) \geq M(v)$ for all $v\in [-A,A]$, one gets for all $v\in [-A,A]$:
$0\leq \lim_{t\to +\infty} \left(\sup_{x\leq ct} \left(M (v) -g (t,x,v)\right)\right)\leq 0$. 
Therefore we conclude
$$\lim_{t\to +\infty} \left( \sup_{x\leq ct}  |M(v) - g(t,x,v)|\right) =0 \quad \hbox{ for all } c<c^*_A \hbox{ and } A>\underline{A}.$$
the conclusion follows from the fact that $\lim_{A\to +\infty} c^*_A=+\infty$. 
\end{proof}

\subsection{Upper bound for the spreading rate \eb{in the gaussian case}}

\label{sec:supersolution}

We construct below supersolutions for  \eqref{eq:kinKPP} when  $V = \R$ and $M$ is a Gaussian distribution. 

\begin{prop}\label{supersolunbounded}
Let $V=\R$ and $M(v) = \frac1{\sigma \sqrt{2 \pi}} \exp \left( - \frac{v^2}{2 \sigma^2} \right) $. For $1\leq b \leq a$ define
\begin{equation}\label{eq:rho and f0}
\rho(t,x) = M \left( \frac{x}{t+a} \right) e^{r(t+a)}\quad \text{and} \quad g^0(x,v) = \frac{1}{b} M\left( \frac{x}{b} \right) M(v) e^{ra}\,.
\end{equation} 
Let $g$ be defined by  
\begin{equation*}
g(t,x,v) = g^0(x-vt,v) e^{-t} + \int_{0}^{t} (1+r) M(v) \rho(s,x-v(t-s)) e^{-(t-s)} ds \,.
\end{equation*}
Then 
$\overline{g}(t,x,v) = \min \left\lbrace M(v) \, , g(t,x,v) \right\rbrace
$
is a supersolution of \eqref{eq:kinKPP}, that is:
\begin{equation*}
\partial_{t} \overline{g} + v \partial_x \overline{g} \geq \left( M(v) \rho_{\overline{g}} - \overline{g} \right) + r \rho_{\overline{g}} \left( M(v)- \overline{g} \right), \quad (t,x,v) \in \R_+ \times \R \times V\, .
\end{equation*}
\end{prop}

\begin{proof}[\bf Proof of Proposition \ref{supersolunbounded}]
We shall prove that $g$ is a supersolution of \eqref{eq:kinKPP}. Indeed, it will follow that $\overline{g}$ is a supersolution since it is the minimum of two supersolutions.
From the Duhamel formula, we deduce that
\begin{equation*}
\partial_{t} g + v \partial_x g + g = (1+r) M(v) \rho ,
\end{equation*}
To prove that $g$ is a subsolution we must prove in fact   that 
\begin{equation*}
(1+r) M(v) \rho  \geq (1+r)M(v)  \rho_{g}  - r \rho_{g} g .
\end{equation*}
This is sufficient to prove that  the inequality $ \rho  \geq  \rho_{g} $ holds true.
Computing the expression of $\rho_{g}$ we obtain
\begin{equation*}
\rho_{g} (t,x) = \underbrace{ \int_V g^0(x-vt,v) e^{-t} dv }_{= A}  + \underbrace{ \int_{0}^{t} (1+r) e^{-(t-s)+r(s+a)} \int_V M(v) M\left( \frac{x - v(t-s)}{s+a} \right) dv ds }_{ = B}
\end{equation*}
We first deal with the estimate of $B$. We claim the following inequality holds true: for all $x \in \R$ and $s \in [0,t]$,
\begin{equation}\label{eq:estimate Gaussian}
\int_V M(v) M\left( \frac{x - v(t-s)}{s+a} \right) dv \leq M\left( \frac{x}{t+a} \right)
\end{equation}
In fact one has
\begin{align*}
\displaystyle \int_V M(v) M\left( \frac{x - v(t-s)}{s+a} \right) dv 
& =  \displaystyle \int_V \frac{1}{2\pi \sigma^2} \exp{ \left( - \frac{v^2}{2\sigma^2}  -  \frac{ \left( \frac{x - v(t-s)}{s+a} \right)^2 }{2 \sigma^2}        \right) } dv\\
& =  \displaystyle \frac{1}{\sqrt{2\pi} \sigma} \frac{s+a}{\left[ (s+a)^{2} + (t-s)^2 \right]^{\frac12}} \exp \left( - \frac{1}{2 \sigma^2} \frac{x^2}{(s+a)^{2} + (t - s)^2} \right)\\
&\leq   \displaystyle \frac{1}{\sqrt{2\pi} \sigma}\exp \left( - \frac{1}{2 \sigma^2} \frac{x^2}{(t+a)^{2}} \right),
\end{align*}
since 
\begin{equation*}
\forall s \in \left[ 0 , t \right], \quad (t+a)^2 \geq (s+a)^2 + (t-s)^2 \geq (s+a)^2.
\end{equation*}
This yields
\begin{align*}
 B (t,x) &\leq  \displaystyle (1+r) \left( \int_{0}^{t}e^{-(t-s)+r(s+a)}ds \right) M\left( \frac{x}{t+a} \right) \\
&=  \displaystyle e^{ra - t} \left( e^{(1+r)t}-1\right) M\left( \frac{x}{t+a} \right) \\
&= \left(  e^{(1+r)t}-1\right) e^{-t+ra} \rho (t,x)e^{-r (t+a)}\\
&= \left( 1- e^{-(1+r)t}\right)  \rho (t,x)
\end{align*}
To estimate $A$, we plug in the formula for $\rho$ \eqref{eq:rho and f0},
\begin{equation*}
\left( \frac{ A }{ \rho } \right) (t,x) = \sqrt{2\pi} \sigma \exp \left( \frac{x^2}{2 \sigma^2 (t+a)^2}  - (1+r) t - r a  \right)  \int_V g^0(x-vt,v) dv \, .
\end{equation*}
We compute the last integral using the formula for the initial condition $g^0$ \eqref{eq:rho and f0},
\begin{equation*}
\int_V g^0(x-vt,v) dv = \frac{1}{\sqrt{2\pi} \sigma} \frac{1}{ \left( t^2 + b^2 \right)^{\frac12}} \exp \left( - \frac{1}{2 \sigma^2} \frac{x^2}{t^2 + b^2} \right) e^{ra},
\end{equation*}
Thus, for all $(t,x) \in \R_+ \times \R$:
\begin{equation}\label{eq:IC}
\left( \frac{ A }{ \rho } \right) (t,x) = \frac{1}{ \left( t^2 + b^2 \right)^{\frac12} } \exp \left(  - \frac{x^2}{2 \sigma^2 (t+a)^{2}} \left[ \frac{(t + a)^2}{t^2 + b^2} - 1  \right] \right)  \exp \left( -(1+r) t  \right) \leq \exp \left( -(1+r) t  \right)
\end{equation} 
as long as $b \geq1$ and $(t+a)^2 \geq  t^2 + b^2$, that is $a \geq  b \geq 1$.
This concludes the proof.
\end{proof}

\begin{proof}[\bf Proof of Theorem \ref{thm:unbounded}]

Let $\e>0$. For all $t \geq 0$, we define the zone 
\begin{equation*}
\Gamma_t = \left\lbrace x \in \R \, \vert \; \vert x \vert\geq \sigma \left( 1 + \eps \right)\sqrt{2r} (t+a)^{3/2} \right\rbrace.
\end{equation*}
From the definition of $\overline{g}$, we deduce that $\overline{g}$ is a supersolution such that $\rho_{\overline{g}} \leq \min \left( 1 , \rho_{g} \right) \leq \min \left( 1 ,  \rho \right)$. However, for all $t>0$ and $x \in \Gamma_t$, we have
\begin{equation*}
\rho(t,x) \leq \frac{1}{\sigma \sqrt{2 \pi}} \exp \left( - \frac{2 r \sigma^2 \left( 1 + \eps \right)^2 (t+a)^3 }{2 \sigma^2 (t+a)^2} + r (t+a)   \right) = \frac{1}{\sigma \sqrt{2 \pi}}  \exp \left( -  r(t+a) \left(   \left( 1 + \eps \right)^2 - 1 \right)    \right).
\end{equation*}
It yields that 
\begin{equation*}
\lim_{t \to + \infty}\left( \sup_{x \in \Gamma_t} \rho(t,x) \right) =  0.
\end{equation*}
\end{proof}

Computations are made easier above since the class of Gaussian distributions is stable by convolution. This is also the case for the class of Cauchy distributions. Therefore we are able to derive an inequality similar to \eqref{eq:estimate Gaussian} in the latter case. Let us comment this case. Because the distribution $M$ has an infinite variance, we learn from \cite{Mellet-Mouhot-Mischler} that the correct macroscopic limit leads to a nonlocal fractional Laplacian operator.
On the other hand, we expect from \cite{Cabre-Roquejoffre, Cabre-Roquejoffre2, Coulon-Roquejoffre} an exponentially fast propagation in the fractional diffusion regime. Similarly as for our previous results,  we can reasonably expect that the spreading rate is faster in the kinetic model than in the macroscopic limit. Therefore we expect a spreading rate faster than exponential. In fact the supersolution that we are able to derive confirms this expectation. 

\eb{In the following Proposition, we construct a supersolution that spreads with rate $\mathcal O(t e^{rt/2})$. However, we leave the complete analysis of spreading in the case of the Cauchy distribution for future work.}

\begin{prop}\label{supersolunbounded-cauchy}
Let $V=\R$ and $M(v) =  \frac{1}{\pi} \frac{\sigma}{\sigma^2+v^2}$. For $1\leq b \leq a - \frac14$, define
\begin{equation}\label{eq:rho and f0 Cauchy}
\rho(t,x) = M \left( \frac{x}{t+a} \right) e^{r(t+a)}\quad \text{and} \quad g^0(x,v) = \frac{1}{b} M\left( \frac{x}{b} \right) M(v) e^{ra}\,.
\end{equation} 
Let $g$ be defined by  
\begin{equation*}
g(t,x,v) = g^0(x-vt,v) e^{-t} + \int_{0}^{t} (1+r) M(v) \rho(s,x-v(t-s)) e^{-(t-s)} ds \,.
\end{equation*}
Then 
$\overline{g}(t,x,v) = \min \left\lbrace M(v) \, , g(t,x,v) \right\rbrace
$
is a supersolution of \eqref{eq:kinKPP}, that is:
\begin{equation*}
\partial_{t} \overline{g} + v \partial_x \overline{g} \geq \left( M(v) \rho_{\overline{g}} - \overline{g} \right) + r \rho_{\overline{g}} \left( M(v)- \overline{g} \right), \quad (t,x,v) \in \R_+ \times \R \times V\, .
\end{equation*}
\end{prop}

\begin{proof}[\bf Proof of Proposition \ref{supersolunbounded-cauchy}]
The proof is the same as for Proposition \ref{supersolunbounded}. We just show the main computations in the case of the Cauchy distribution. 
To prove \eqref{eq:estimate Gaussian} we use the residue method, 
\begin{align*}
\displaystyle \int_V M(v) M\left( \frac{x - v(t-s)}{s+a} \right) dv & =   \displaystyle \int_V \frac{\sigma^2}{\pi^2} \frac{1}{\sigma^2 +v^2} \cdot \frac{1}{\sigma^2 + \left( \frac{x - v(t-s)}{s+a} \right)^2}   dv\\
& =  \displaystyle \frac{\sigma}{\pi} \frac{(s+a)(t+a)}{x^2 + \sigma^2(t+a)^2} \\
& \leq  M\left( \frac{x}{t+a} \right)  
\end{align*}
The analog computation for proving \eqref{eq:IC} goes as follows. First we have
\begin{equation*}
\left( \frac{ A }{ \rho } \right) (t,x) = {\pi} \left( 1 + \left( \frac{x}{t+a} \right)^2 \right) \exp \left( - (1+r) t - r a  \right)  \int_V f^0(x-vt,v) dv 
\end{equation*}
Thanks to the expression of the initial condition, we compute the latest integral:
\begin{equation*}
\int_V f^0(x-vt,v) dv = \frac{1}{\pi} \frac{t+b}{ x^2 + (t+b)^2} e^{ra},
\end{equation*}
Thus, 
\begin{equation*}
\left( \frac{ A }{ \rho } \right) (t,x) = \frac{t+b}{(t+a)^2} \frac{x^2 + (t+a)^2}{ x^2 + (t+b)^2} \exp \left( -(1+r) t  \right) \leq \exp \left( -(1+r) t  \right)
\end{equation*} 
which holds true if $b \geq 1$ and $a \geq b + \frac14$. 
\end{proof}

\eb{
\subsection{Lower bound for the spreading rate  in the gaussian case}

We construct below subsolutions for  \eqref{eq:kinKPP} when  $V = \R$ and the distribution $M$ is a Gaussian distribution. Contrary to the previous Section \ref{sec:supersolution}, the strategy does not rely on a specific computational trick ({\em i.e.} Gaussian distributions are stable by convolution). We rather build a typical subsolution based on the dispersion property of the kinetic transport-scattering operator. This construction heavily relies on preliminary results obtained in \cite{Bouin-Calvez-Grenier-Nadin}. We shall motivate the construction of the subsolution based on these ideas. 

The first fact is that the subsolution we build here is essentially not regular. It is discontinuous along the line $x = vt$, so that it is not affected by the free transport operator. It is also discontinuous along the line $v=-K$ for some (large) $K$. However, this does not cause any further trouble due to the absence of derivatives with respect to velocity.  Moreover the necessary truncation at a certain level $\gamma \in (0,1)$ yields $C^1$ discontinuities. This is not a problem since the PDE is of order one, as opposed to classical reaction-diffusion equations for which such a rough troncature is not possible when seeking subsolutions due to the presence of second-order derivatives. Of course we pay much attention to the nonlocal contributions (integral with respect to the velocity) where this truncature causes  additional difficulties.

The second point to highlight is that long-range dispersion happens via the free transport operator and the redistribution with respect to the velocity (scattering). It is obviously a matter of (small) densities having large velocities. This is exemplified when noticing that the function defined by 
\[ g_2(t,x,v) = \exp\left( - \frac xv\right) M(v)  \, , \quad \text{if }\;   v> \frac x t\, ,\] and zero elsewhere, is a solution of 
\[ \partial_t g_2 + v \partial_x g_2 + g_2 = 0\, , \] and thus a subsolution of \eqref{eq:kinKPP} under the condition that $g \leq M $ everywhere. 
This branch of the solution (restricted to $v> x/t$) will contribute to dropping the mass after redistribution through the nonlocal "source" term $(1 + r) M(v) \rho_g$ in the area $-K < v< x/t$. There is some technical issue due to the fact that $g_2$ is unbounded for $x<0$. We will circumvent by truncating the density. 

Another technical issue stems from the fact that we require negative velocities (up to $v>-K$ for $K$ large enough) in order to maintain enough local redistribution through the nonlocal source term $(1 + r) M(v) \rho_g$. This yields an artificial linear transport to the left side (with velocity $-K$). Nonetheless, as superlinear spreading is expected, this backward linear transport term will not affect the conclusion.  Our strategy consists in working in the moving frame $y = x + Kt$.


We are now in position to define a proper subsolution. Let $K,L$ be two positive (large) bounds on the velocity. Let $\gamma \in (0,1)$ be a truncation level. We stress that the subsolution will automatically satisfy
\[ \ug \leq \gamma M(v)\, . \]
Therefore we are led to find such a function $\ug$ verifying the following inequality,
\begin{equation} \label{eq:subsol} 
\begin{cases}
\partial_t \ug + v\partial_x \ug + \ug \leq (1 + (1 - \gamma) r )  M(v) \rho_g \, , \medskip \\
\ug(0,x,v) = \gamma M(v) {\bf 1}_{x < A} 
\end{cases}
\end{equation}
%
As mentioned above, we shall set $\ug = 0$ for $v<-K$. 
To define the subsolution for $v > -K$, we set the problem in the moving frame $y = x + Kt$. 
Equation \eqref{eq:subsol} writes 
\begin{equation} \label{eq:subsol*} 
\begin{cases}
\partial_t \ug + (v+K) \partial_y \ug + \ug \leq (1 + (1 - \gamma) r )  M(v) \rho_g \, ,  \medskip \\
\ug(0,y,v) = \gamma M(v) {\bf 1}_{y < A} 
\end{cases}
\end{equation}
\begin{figure}
 \begin{center}
 \includegraphics[width = .8\linewidth]{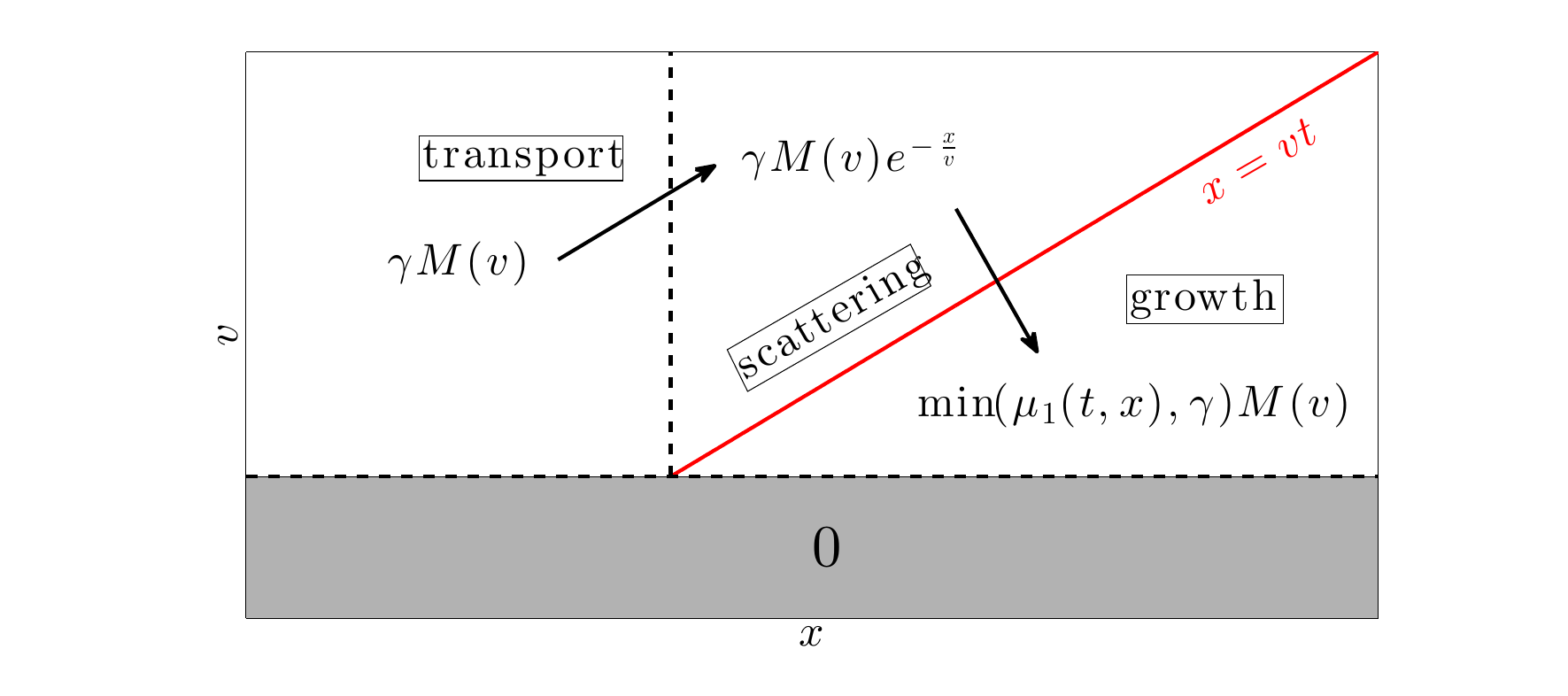}
 \caption{\label{fig:subsol} Schematic view of the subsolution $\ug$. It is defined piecewise. We have set $K = 0$ for the sake of clarity. The mechanism which drives the subsolution can be described as follows. The free transport operator sends very few particles with very high velocity at the edge of the front. They are redistributed, and their density grows exponentially fast. The mass in the lower branch $\{v<x/t\}$, $\mu_1(t,x)$, is computed analytically.}
\end{center}
\end{figure}
We define $\ug$ piecewise:
\begin{itemize}
\item For $y \geq 0$, 
\begin{equation}\label{eq:def ug 1}
\ug(t,y,v) = \begin{cases} \ds\gamma \exp\left( - \frac {y}{v+K}\right) M(v)   \, , &\quad \text{if }\quad  v + K> \dfrac y t \medskip \\
\min\left(\mu_1(t,y),  \gamma\right) M(v)\, , & \quad  \text{if }\quad  0< v + K< \dfrac y t\medskip \\
0\, , & \quad  \text{if }\quad  v + K < 0 
\end{cases}
\end{equation}
\item For $y \leq 0$, 
\begin{equation}\label{eq:def ug 2}
\ug(t,y,v) = \begin{cases} \ds\gamma M(v)   \, , &\quad \text{if }\quad  v + K> 0 \medskip \\
0\, , & \quad  \text{if }\quad  v + K< 0 
\end{cases}
\end{equation}
\end{itemize}
For a schematic view of the subsolution and the growth-dispersion process, see Figure \ref{fig:subsol}.  
The partial mass contained in the mid-zone $\left(-K, y/t - K \right)$ is denoted by $\mu_1(t,y)$. It is defined as the solution to the following ODE,
\begin{equation}
 \partial_t \mu_1 + \mu_1 = (1 + (1 - \gamma) r) \left(  \min\left(\mu_1, \gamma\right) \int_{-K}^L M(v)\, dv + \mu_2 \right) \, ,\label{eq:mu1}
\end{equation}
with the initial datum 
\begin{equation} \mu_1(0,y) =  \gamma  {\bf 1}_{y < A}\, . \label{eq:mu1(0)} \end{equation}
Finally, the source term $\mu_2$ is defined as the partial mass contained in the branch $v> y/t - K$:
\begin{equation}
 \mu_2(t,y) = \ds \gamma \int_{y/t - K }^{\infty} \exp\left(- \frac{y}{v+K}\right) M(v) \,  dv\, .
 \label{eq:mu2} \end{equation}
We stress out that there is a minor discrepancy between the requirement for being a subsolution \eqref{eq:subsol} and the definition of \eqref{eq:mu1}. Namely, the integral contribution runs over $v\in (-K,L)$, although it should naively be $v\in (-K,y/t - K)$. However it is mandatory for the sequel that $\mu_1$ is nonincreasing with respect to $y$, which is not obvious if $L$ is replaced with $y/t - K$ in \eqref{eq:mu1}. Note that $\mu_2$ is indeed nondecreasing with respect to $y$, so that $\mu_1$ defined by \eqref{eq:mu1}-\eqref{eq:mu1(0)} is clearly nondecreasing \eb{with respect to $y$} as well. A simple way to eliminate this discrepancy is to guarantee that $\mu_1(t,y) \geq  \gamma $ when $L> y/t - K$, which is the wrong sign of the discrepancy. This is somehow expected since we know {\em a posteriori} that the front is located in the region $y = \mathcal O(t^{3/2})$, such that the unsaturated area is such that $L < y/t - K$ for large time. For small time, it will be guaranteed by tuning the range of the initial datum, namely the parameter $A$.

The remainder of this Subsection is organized as follows: we first establish some technical estimates of $\mu_2$. Then we deduce that $L> y/t - K$ implies $\mu_1\geq \gamma$. As a consequence, we establish that $\ug$ is a subsolution of \eqref{eq:kinKPP}. Finally, the technical estimates are used again to prove that $\ug$ (in fact, $\mu_1$) exhibits superlinear propagation with the expected scaling $y = \mathcal O(t^{3/2})$.

We  introduce  the modified growth rate 
\[\tilde r = (1 + (1 - \gamma) r) \int_{-K}^L M(v)\, dv - 1\, . \]
Note that $\tilde r < r$, and it can be chosen arbitrary close to $r$, by varying $\gamma, K , L$. 

\begin{lem}
The function $\mu_1$ is given by the Duhamel  formula in the area $\{\mu_1 < \gamma\}$:
\begin{equation} 
\mu_1(t,y) =  e^{\tilde r t} \mu_1(0,y)+ (1 + (1 - \gamma) r)\int_0^t e^{\tilde r (t-s)} \mu_2(s,y) \, ds\, . \label{eq:mu1bis}
\end{equation} 
\end{lem}

\begin{proof}
We notice that $\mu_1$ is nondecreasing with respect to $t$, since the source term $\mu_2$ is nonnegative. Therefore the formula \eqref{eq:mu1bis} is valid up to $\mu_1 = \gamma$.  
\end{proof}

\begin{lem}
The following estimate holds true,
\begin{equation}\label{eq:estimate mu2}
\mu_2(t,y) \geq \begin{cases}
\displaystyle \frac{1 }{r_1(t,y)}\exp\left(-\frac1{2\sigma^2 } \left(\frac{y}{t} - K\right)^2  - t\right)  \,, & \quad \text{if }\quad  \dfrac{y}{t} > v^*(y)\medskip\\
\displaystyle \frac{1  }{ r_2(y) }\exp\left(- \frac{y}{v^*(y)} - \frac{(v^*(y)-K)^2}{2\sigma^2}\right) \,, & \quad \text{if }\quad  \dfrac{y}{t} < v^*(y).
\end{cases}
\end{equation}
where $v^*(y)$ is a velocity satisfying $v^*(y) \sim_{y \to + \infty} \sigma^{2/3}y^{1/3}$ and $r_1$ and $r_2$ are lower order corrections (as compared to the exponential decay), and are described below in the proof.
\end{lem}

\begin{proof}
Before we start with the technical estimates, let us explain briefly why \eqref{eq:estimate mu2} is very much expected. We consider $K = 0$ for the sake of simplicity. The function $-\log\left( e^{-\frac xv} M(v) \right) = \frac xv + \frac{v^2}{2\sigma^2} + \frac12\log ( 2\pi\sigma^2 ) $ admits a global minimum with respect to $v>0$, attained at $v^* = \sigma^{2/3}x^{1/3}$. It should be discussed whether this minimum lies in the area $v>\frac x t$ or not. In any case, the integral $\mu_2$ is close to the value of the corresponding exponential maximum, up to lower order corrections included in $r_{1,2}$.
\begin{figure}
 \begin{center}
 \includegraphics[width = .6\linewidth]{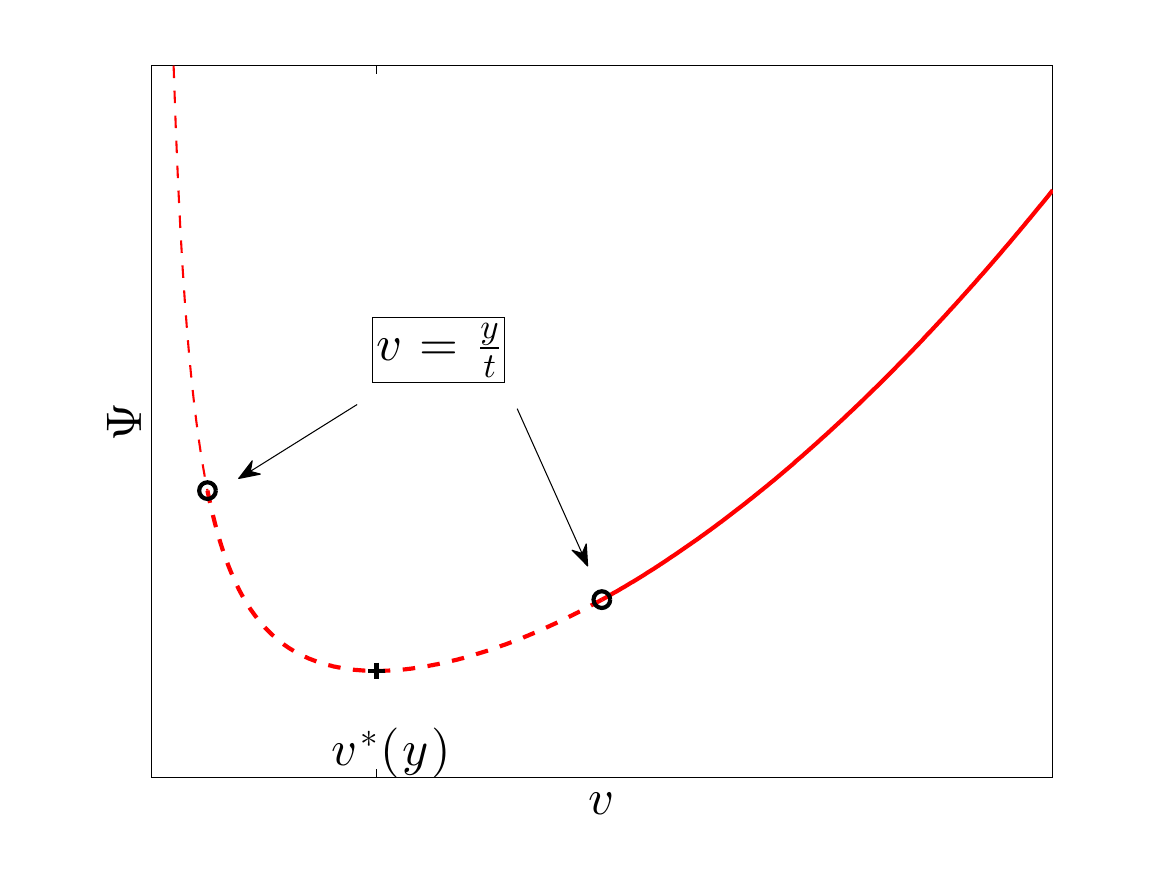}
 \caption{\label{fig:int-mu2} Scheme of the proof of the estimate \eqref{eq:estimate mu2}: the relative positions of $\frac yt$ and $v^*(y)$ have to be discussed. In any case, we restrict to the nondecreasing branch $v> v^*(y)$ to estimate $\mu_2(y)$, and we apply a suitable change of variables in this branch.}
\end{center}
\end{figure}

We consider now a general $K$. To estimate $\mu_2$, let us rewrite
\begin{equation*}
\mu_2(t,y) = \ds \frac{\gamma}{\sigma \sqrt{2 \pi}} \int_{\frac{y}{t}}^{\infty} e^{- \frac{y}{v}} e^{-\frac{(v-K)^2}{2 \sigma^2}} \,  dv\,. 
\end{equation*}
We observe that $\Psi := v \mapsto \frac yv + \frac{(v-K)^2}{2\sigma^2}$ has a minimum on $\R^+$, attained at the velocity $v^*(y)$ given by the nonnegative root of
\begin{equation}\label{eq:v*}
v^2(v-K) = \sigma^2 y,
\end{equation}
see Figure \ref{fig:int-mu2}. As a byproduct of this first order condition, we get \eb{directly} that $v^*(y) \geq K$. We get also that 
\begin{equation}\label{eq:equivv*}
v^*(y) \sim \sigma^{2/3}y^{1/3}\,, \quad \textrm{and accordingly, } \quad \Psi(v^*(y)) \sim \frac32 \left( \frac{y}{\sigma} \right)^{2/3}\, , \quad \text{as}\, y\to +\infty\, .
\end{equation}
This equivalents will be of crucial importance in the following Lemmas to estimate properly the propagation. 
We can now uniquely define the change of variables
$\varphi : \left( \Psi\left(v^*(y)\right) , + \infty \right) \rightarrow \left( v^*(y) , + \infty \right)
$
defined by
\begin{equation}\label{def:Phi}
\frac{\left( \varphi(u) - K \right)^2}{2\sigma^2} + \frac{y}{\varphi(u)} = u.
\end{equation}
Moreover, $\varphi$ is increasing as the inverse function of an increasing function. Multipying by $\varphi(u)$ and differentiating \eqref{def:Phi} yields 
\begin{equation*}
\forall u \in \left( \Psi\left(v^*(y)\right) , + \infty \right) , \quad \varphi'(u) = \frac{\varphi(u)}{\frac{1}{2\sigma^2} \left( 3\varphi^2(u) - 4K \varphi(u) +K^2 \right)  - u}.
\end{equation*}
We distinguish naturally between two cases, depending on the relative positions of $\frac yt$ and $v^*(y)$. 

\medskip

\noindent{\bf \# Step 1: The case $\frac{y}{t} \geq v^*(y)$.} 
We directly apply the  change of variables:
\begin{equation*}
\ds \frac{\gamma}{\sigma \sqrt{2 \pi}} \int_{\frac{y}{t}}^{\infty} e^{- \frac{y}{v}} e^{-\frac{(v-K)^2}{2 \sigma^2}} \,  dv = \frac{\gamma}{\sigma\sqrt{2\pi}} \int_{\Psi\left( \frac yt \right) }^{\infty} e^{-u} \varphi'(u)\,  du\,,
\end{equation*}
We need some estimate from below of $\varphi'(u)$. We first deduce from $\varphi'(u) \geq 0$ that 
\begin{equation*}
3\varphi^2(u) - 4K \varphi(u) +K^2 \geq 2\sigma^2 u\,.
\end{equation*}
It yields that necessarily
\begin{equation*}
\varphi(u) \leq \frac23 K - \sqrt{\frac{K^2}{9} + \frac{2\sigma^2}{3} u}\,, \quad \textrm{ or } \quad \varphi(u) \geq \frac23 K + \sqrt{\frac{K^2}{9} + \frac{2\sigma^2}{3} u}\,.
\end{equation*}
The first alternative is impossible since $\varphi(u) \geq v^*(y)\geq  K$.
On the other hand, one deduce from the very definition \eqref{def:Phi} and  $\varphi(u) \geq  K$ that
\begin{equation*}
\begin{array}{lcl}
\ds \frac{1}{2\sigma^2} \left( 3 \varphi^2(u) - 4K \varphi(u) +K^2 \right)  - u  & = & \ds \frac{3}{2\sigma^2} \left(  \left( \varphi(u) - \frac{2K}{3} \right)^2 - \frac{K^2}{9} \right)  - u \, , \\
& \leq & \ds \frac{3}{2\sigma^2} \left( \varphi(u) - K \right)^2  - u  \, ,\\
& = & 3u - 3 \frac{y}{\varphi(u)} - u \,, \\
& \leq & 2u.
\end{array}
\end{equation*}
We deduce that 
\begin{equation*}
\varphi'(u) \geq  \frac{\frac{2}{3}K + \sqrt{\frac{K^2}{9} + \frac{2\sigma^2}{3} u}}{2u} \geq \left(\frac{2\sigma^2}{3}\right)^\frac12 \frac{1}{2 \sqrt{u}}\, .
\end{equation*}
We obtain as a consequence,
\begin{equation*}
\mu_2(t,y) \geq \frac{\gamma}{\sigma\sqrt{2\pi}} \int_{\Psi\left( \frac yt \right) }^{\infty}  e^{-u} \left(\frac{2\sigma^2}{3}\right)^\frac12 \frac{du}{2 \sqrt{u}} = \frac{\gamma }{\sqrt{3\pi}}\int_{\left( \Psi\left( \frac yt \right)\right)^{\frac12}}^{\infty}   e^{-v^2} dv. 
\end{equation*}
Next, we apply a quantitative estimate for the remainder of the gaussian integral \cite[p. 298]{Abram}:
\begin{equation}\label{eq:erf}
\forall x \geq 0, \quad \int_{x}^{\infty} e^{-x^2} dx > \frac{e^{-x^2}}{x + \sqrt{x^2 + 2}}\, .
\end{equation}
Consequently we obtain
\begin{equation*}
\mu_2(t,y)   \geq   \ds \frac{\gamma}{\sqrt{3\pi}}   \frac{e^{- \Psi\left( \frac yt \right)}}{\left( \Psi\left( \frac yt \right)\right)^{\frac12} + \left( \Psi\left( \frac yt \right) + 2\right)^{\frac12}}.
\end{equation*}

\medskip

\noindent{\bf \# Step 2: The case $\frac{y}{t} \leq v^*(y)$.} There, we simply neglect the decreasing part of $\Psi$ (see Figure \ref{fig:int-mu2}). The result is 
 a direct consequence of the  previous calculation. Indeed, we have:
\begin{equation}
\mu_2(t,y) = \ds \frac{\gamma}{\sigma \sqrt{2 \pi}} \left( \int_{\frac{y}{t}}^{v^*(y)} e^{- \frac{y}{v}} e^{-\frac{(v-K)^2}{2 \sigma^2}} \,  dv\, + \ds \int_{v^*(y)}^{\infty} e^{- \frac{y}{v}} e^{-\frac{(v-K)^2}{2 \sigma^2}} \,  dv\,\right). 
\end{equation}
After neglecting the first contribution, and following the same lines as in Step 1 for the second integral,  we get eventually: 
\begin{equation*}
\mu_2(t,y) \geq \ds \frac{\gamma}{\sqrt{3\pi}}  \frac{e^{- \Psi\left( v^*(y) \right)}}{\left( \Psi\left( v^*(y) \right)\right)^{\frac12} + \left( \Psi\left( v^*(y) \right) + 2 \right)^{\frac12}}.
\end{equation*}
\end{proof}

\begin{lem}\label{lem:mu1 from below}
There exists $A_0$  such that for all $A\geq A_0$, the following estimate holds true,
\begin{equation*}
\left( \forall (t,y) \in \R^+ \times \R^+\right) \quad \frac{y}{t} - K < L \, \Longrightarrow  \, \mu_1(t,y) \geq \gamma\,.
\end{equation*}
\end{lem}

\begin{proof}
Because $\mu_1$ is nonincreasing with respect to $y$, it is sufficient to prove that 
\begin{equation*}
\forall t \in \R^+, \quad \mu_1(t,(K+L)t) \geq \gamma.
\end{equation*}
We recall the definition of $\mu_1$ in the area $\{\mu_1<\gamma\}$,
\begin{equation}
\label{eq:def mu1(t,Kt)}
\mu_1(t,(K+L)t) =  \gamma e^{\tilde r t}  {\bf 1}_{(K+L)t < A} + (1 + (1 - \gamma) r)\int_0^t e^{\tilde r (t-s)} \mu_2(s,(K+L)t) \, ds\, . 
\end{equation}
We observe that for $t < \frac{A}{K+L}$, the condition $\mu_1(t,(K+L)t) \geq \gamma$ is  fulfilled due to the initial datum. Then we estimate the integral term in the r.h.s. of \eqref{eq:def mu1(t,Kt)}. The following estimate is crucial since it moreless contains the superlinear propagation behavior. Let $\alpha \in (0,1)$ to be chosen later. We have 
\begin{equation*}
\int_0^t e^{\tilde r (t-s)} \mu_2(s,(K+L)t) \, ds \geq \mu_2( \alpha t,(K+L)t)\int_{\alpha t}^t e^{\tilde r (t-s)} ds.
\end{equation*}
As a consequence, when $t$ is large enough such that 
\begin{equation*}
\frac{K+L}{\alpha} = \frac{(K+L)t}{\alpha t} \leq v^*\left( (K+L)t \right)\, ,
\end{equation*}
we have
\begin{equation}\label{eq:prop}
\mu_1(t,(K+L)t) \geq \mu_2( \alpha t,(K+L)t)\int_{\alpha t}^t e^{\tilde r (t-s)} ds \geq \frac{1}{\tilde r} \left( e^{ \tilde r (1 - \alpha)t } - 1 \right) \frac{\gamma}{\sqrt{3\pi}}   \frac{e^{- \Psi\left( v^*((K+L)t) \right)}}{r_2( (K+L)t )}
\end{equation}
Thus, it is enough to guarantee that the following estimate holds true, 
\begin{equation*}
\frac{1}{\tilde r} \left( e^{ \tilde r (1 - \alpha)t } - 1 \right) \frac{1}{\sqrt{3\pi}}  \frac{e^{- \Psi\left( v^*((K+L)t) \right)}}{r_2( (K+L)t )} \geq 1, 
\end{equation*}
for $t$ large enough, say $t> T_0$. It is indeed the case  since \eqref{eq:equivv*} implies that
\[ \log \left(\mu_1(t,(K+L)t)\right) \geq \tilde r (1 - \alpha) t + O (t^{3/2} ) + O(\log t)\, . \]
The conclusion is straightforward, provided we choose $A$ large enough such that $T_0 = \frac{A}{K+L}$ satisfies 
\[v^*(A)>\frac{K+L}\alpha\, .\]
Finally we notice that $\alpha$ is still a free parameter in the range $(0,1)$. It will be fixed in the next Lemma when optimizing the propagation behavior. 
%
\end{proof}


%

\begin{thm}
Let the constants $K,L,\gamma$ chosen as above. Let $\alpha  < \frac{\tilde r}{\tilde r + 2} $, and choose $A$ accordingly.
The function $\ug$ defined by \eqref{eq:def ug 1}-\eqref{eq:def ug 2} is a subsolution of \eqref{eq:kinKPP}. Moreover it exhibits a superlinear spreading with rate $\mathcal O(t^{3/2})$. More precisely, the point $y(t)$ such that  $\mu_1(t,y(t)) = \frac\gamma2$ is such that  $y(t) \geq \sigma (\alpha t)^{3/2}$ for $t$ sufficiently large.
\end{thm}

\begin{proof}
We first observe that for all $y>0$ we have $v^*(y)\geq \sigma^{2/3}y^{1/3}$ \eqref{eq:v*}. On the other hand, there exists $Y_0$ such that for $y\geq Y_0$ we have 
\begin{equation*}
\Psi\left( v^*(y) \right) \leq 2 \sigma^{-2/3}y^{2/3}\, .
\end{equation*}
We define the zone
\begin{equation*}
\Y_t = \left\{ y \, : \, Y_0 \leq y \leq \sigma \left( \alpha t \right)^{3/2}\right\}\, .
\end{equation*}
For $y\in \Y_t$ we have immediately 
\begin{equation*}
 \frac{y}{\alpha t} \leq \sigma^{2/3}y^{1/3} \leq v^*\left( y \right)\, ,
\end{equation*}
where we have used that $y\geq Y_0$ to justify the last inequality. 
We recall the estimation of $\mu_1$ \eqref{eq:prop} which holds true for $y\in \Y_t$: 
\begin{align*}
\mu_1(t,y) & \geq \mu_2( \alpha t,y)\int_{\alpha t}^t e^{\tilde r (t-s)} ds \\
& \geq  \frac{1}{\tilde r} \left( e^{ \tilde r (1 - \alpha)t } - 1 \right) \frac{\gamma}{\sqrt{3\pi}}   \frac{e^{- \Psi\left( v^*(y) \right)}}{r_2( y )} \\
&\geq  \frac{1}{\tilde r} \left( e^{ \tilde r (1 - \alpha)t } - 1 \right) \frac{\gamma}{\sqrt{3\pi}} \frac{e^{- 2 \sigma^{-2/3}y^{2/3}}}{r_2( y )}.
\end{align*}
This yields
\[ (\forall y\in \Y_t)\quad \log( \mu_1(t,y) ) \geq \left( \tilde r (1 - \alpha) - 2 \alpha \right) t + O(\log t)\, .  \]
As a consequence, choosing $\alpha \in (0,1)$ such that 
\[ \alpha < \frac{\tilde r}{\tilde r + 2}\, , \] 
ensures that for sufficiently large times, $\mu_1(t,y) \geq \gamma$, and thus the front has already passed through $\Y_t$. 
\end{proof}

\medskip
\paragraph*{Acknowledgements.} The authors wish to thank Jimmy Garnier and Emmanuel Grenier for enlightening discussions concerning the correct spreading rate in the gaussian case. 
}

\section*{Appendix}

We give in this \eb{Appendix} the proof of Propositions \ref{Cauchypbm} and \ref{prop-mp}. Well-posedness relies on a fixed point argument which is also used for the comparison principle. We first state two Lemmas.


\begin{lem}\label{lem:eulin}
 Let $a,b \in \mathcal{C}^0_b \left(\R_+ \times \R\times V\right)$ and $g^0 \in \mathcal{C}^0_b \left(\R , L^1( V)\right)$. 
Then  there exists a unique function $g \in \mathcal{C}^0_b \left( \R_+\times \R, L^1( V)\right)$ such that 
\begin{equation} \label{eq:linab} \left\{\begin{array}{lcl}
\partial_t g + v\partial_x g + a(t,x,v)g = b(t,x,v) \rho_g \quad & \hbox{in } &\R_+ \times \R\times V,\\
g(0,x,v) = g^0 (x,v) \quad & \hbox{in } &\R\times V,
                                         \end{array}\right. \end{equation}
in the sense of distributions. 
This solution also satisfy the Duhamel formula:
\begin{multline} \label{eq:linab2} g(t,x,v) = g^0 (x-vt,v) e^{-\int_0^t a \left(s,x-(t-s)v,v\right)ds} \\+ \int_0^t e^{-\int_s^t a(\tau,x-(t-\tau)v,v)d\tau} b\left(s,x-v(t-s),v\right) \eb{\rho_g \left(s,x-v(t-s)\right)}\, ds\, .\end{multline}
Moreover, if $b\geq 0$ and $g^0\geq 0$, then $g \geq 0$ in $\R_+ \times \R\times V$. 
\end{lem}

\begin{proof}[\bf Proof of Lemma \ref{lem:eulin}]
For $T>0$ we define the operator
\begin{equation} \label{eq:defAT}\begin{array}{rrcl} A_T : & \mathcal{C}^0_b \left((0,T)\times \R, L^1 (V)\right) &\to& \mathcal{C}^0_b \left((0,T)\times \R, L^1 (V)\right) \smallskip \\
   & g&\mapsto& \widetilde{g}
  \end{array}\end{equation}
where
\begin{multline}
\widetilde{g} (t,x,v) = g^0 (x-vt,v) e^{-\int_0^t a \left(s,x-(t-s)v,v\right)ds} \\+ \int_0^t e^{-\int_s^t a(\tau,x-(t-\tau)v,v)d\tau} b\left(s,x-v(t-s),v\right) \eb{\rho_g \left(s,x-v(t-s)\right)}\, ds\, .
\end{multline}
Take $g_1,g_2 \in  \mathcal{C}^0_b \left((0,T)\times \R, L^1 (V)\right)$ and define $\widetilde{g}_1 = A_T g_1$ and $\widetilde{g}_2 = A_T g_2$. Assume that $a\not\equiv 0$ over $(0,T) \times \R\times V$. 
For all $(t,x) \in (0,T) \times \R$, one has:
\begin{align*}
&\int_V \left|\widetilde{g}_1 (t,x,v) -\widetilde{g}_2 (t,x,v)\right|\,dv \\
 &\leq \int_V \int_0^t e^{-\int_s^t a(\tau,x-(t-\tau)v,v)d\tau} b\left(s,x-v(t-s),v\right) \left|\rho_{g_1} \eb{\left(s,x-v(t-s)\right)} - \rho_{g_2} \left(s,x-v(t-s)\right)\right|\, dvds\\
&\leq \int_0^t e^{(t-s) \|a\|_{L^\infty}} \|b\|_{L^\infty} ds \times \sup_{(t,x) \in (0,T) \times \R} \int_V |g_1(t,x,v) - g_2 (t,x,v)|\,  dv\\
&\leq  \frac{1}{\|a\|_{L^\infty}} \left(e^{T \|a\|_{L^\infty}}-1\right)  \|b\|_{L^\infty} \times \sup_{(t,x) \in (0,T) \times \R} \int_V |g_1(t,x,v) - g_2 (t,x,v)|\,  dv\, .
  \end{align*}
Hence, there exists $T_0>0$ such that for all $T\in (0,T_0)$, $A_T$ is a contraction over $\mathcal{C}^0_b \left((0,T)\times \R, L^1 (V)\right)$. 
If $a\equiv 0$ on $(0,T) \times \R\times V$, then such an estimate can be derived similarly. 
Hence, $A_T$ admits a unique fixed point, which satisfies
\eqref{eq:linab2} over $(0,T) \times \R\times V$. This gives the local existence and uniqueness of the solution of \eqref{eq:linab2}. 
Moreover, as $T_0$ does not depend on the initial datum $g^0$, the global existence follows. 

If $b\geq 0$ and $g^0\geq 0$, then $A_T$ preserves the cone of nonnegative functions and thus applying the fixed point theorem in this cone, we get the nonnegativity of $g$. 
%
\end{proof}

\begin{lem} \label{lem:posA_T} Assume that $b$ is everywhere positive and that $V$ is an interval. 
Then if $g^0 \in \mathcal{C}_b^0 (\R_+ \times \R\times V)$ is nonnegative and if
there exists $(x_0,v_0)\in \R\times V$ such that $g^0(x_0,v_0)>0$, letting $g$ the unique solution of \eqref{eq:linab}, one has 
$g (t,x,v) >0$ for all $(t,x,v) \in \R_+ \times \R\times V$ such that $|x-x_0|< v_{\rm max} t$. 
\end{lem}

\begin{proof}[\bf Proof of Lemma \ref{lem:posA_T}]
First, assume by contradiction that there exists $(t,x) \in \R_+ \times \R$ such that $\rho_g (t,x) = 0$, with $|x-x_0|<v_{\rm max}t$. 
Then integrating \eqref{eq:linab2} over $V$, one gets 
\begin{multline*}
0 = \rho_g (t,x) = \int_{v\in V} g^0 (x-vt,v) e^{-\int_0^t a(s,x-(t-s)v,v)ds}dv \\+ \int_{v\in V}\int_0^t e^{-\int_s^t a(\tau,x-(t-\tau)v,v)d\tau} b\left(s,x-v(t-s),v\right) \rho_g \left(s,x-v(t-s)\right)\,dsdv.
\end{multline*}
Hence, $\rho_g \left(s,x-v(t-s)\right)=0$ for all $v\in V$ and $s\in (0,t)$. Letting $s\to 0$, one gets $\rho_g \left(0,x-vt\right)=0$ for all $v\in V$. As $|x-x_0|< v_{\rm max} t$ 
and $V$ is an interval, one can take $v\in V$ such that $x-vt = x_0$, leading to $\rho_g (0,x_0)=0$. This is a contradiction since, as $g$ is continuous, nonegative and $g(0,x_0,v_0)>0$, one 
has $\rho_g (0,x_0)>0$. Hence $\rho_g (t,x,v) >0$ for all $(t,x,v) \in (0,T) \times R\times V$ such that $|x-x_0|<v_{\rm max}t$.

Next, as
$$g (t,x,v) = g^0 (x-vt,v) e^{-\int_0^t a(s,x-(t-s)v,v)ds} + \int_0^t e^{-\int_s^t a(\tau,x-(t-\tau)v,v)d\tau} b\left(s,x-v(t-s),v\right) \rho_g \left(s,x-v(t-s)\right)\,ds\,,$$
it follows from the first step that $g(t,x,v)>0$ as soon as there exists $s\in (0,t)$ such that $|x-x_0-v(t-s)| < v_{\rm max}s$, which also reads:
$|x-x_0| < v_{\rm max} t$. 
\end{proof}

\begin{proof}[\bf Proof of Proposition \ref{prop-mp}.]
Define $w=g_1 - g_2$. As in the proof of Lemma 6 in \cite{Cuesta12}, we first remark that this function satisfies
\begin{equation} \label{eq:mph}
  \partial_t w+  v \partial_x w + \left(1+r\rho_{g_1}\right)w \geq  \left( M(v) + r \left( M(v) - g_2\right) \right) \rho_w \hbox{ in } \R_+\times \R\times V,
 \end{equation}
with $w(0,x,v)\geq 0$ for all $(x,v) \in\R\times V$. We define $a = 
1+r\rho_{g_1}$ and $b =  M(v) + r \left( M(v) - g_2\right) $. Writing the integral formulation as in the proof of Lemma \ref{lem:eulin} gives
\begin{equation*} 
w \left(t,x,v\right)  \geq
\int_0^t e^{-\int_s^t a(\tau,x-(t-\tau)v,v)d\tau} b\left(s,x-v(t-s),v\right) \rho_w \left(s,x-v(t-s)\right)\,ds\,,
\end{equation*}
and thus $w \geq A_T w$ in $(0,T) \times \R\times V$ for some operator $A_T$ which is monotone and contractive when $T$ is small enough. 
It follows that $w \geq A_T^n w$ for all $n\geq 1$. Since $A_T$ is contractive  the sequence $\left(A_T^n w\right)_n$ converges to $0$.  
We conclude that $w\geq 0$, meaning that $g_1 \geq g_2$. 

Next, assume that $\inf_{V} M>0$, $V$ is an interval, and that there exists $(x_0,v_0)$ such that $g_2 (0,x_0,v_0)> g_1 (0,x_0,v_0)$. 
We can follow the proof of Lemma \ref{lem:posA_T}, where $b$ defined above is positive everywhere. We deduce that 
$w(t,x,v) >0$ as soon as $|x-x_0| < v_{\rm max} t$. 
\end{proof}

\end{document}